\documentclass{amsart}
\usepackage{tikz}
\usepackage{xcolor}
\usepackage{amssymb,latexsym,amsmath,amsthm,extarrows}
\usepackage{graphicx,mathrsfs}

\numberwithin{equation}{section}

\newtheorem{theorem}{Theorem}[section]
\newtheorem{lemma}[theorem]{Lemma}

\newtheorem{proposition}[theorem]{Proposition}
\newtheorem{remark}[theorem]{Remark}

\newtheorem{definition}[theorem]{Definition}

\newtheorem{conjecture}[theorem]{Conjecture}

\newcommand{\al}{\alpha}
\newcommand{\be}{\beta}

\newcommand{\Ga}{\Gamma}
\newcommand{\de}{\delta}

\newcommand{\e}{\varepsilon}

\newcommand{\la}{\lambda}

\newcommand{\La}{\Lambda}
\newcommand{\si}{\sigma}

\newcommand{\vp}{\varphi}
\newcommand{\om}{\omega}

\newcommand{\cv}{\mathcal V}
\newcommand{\cu}{\mathcal U}
\newcommand{\cq}{\mathcal Q}
\newcommand{\cp}{\mathcal P}
\newcommand{\cl}{\mathcal L}

\newcommand{\co}{\mathcal O}

\newcommand{\cf}{\mathcal F}

\newcommand{\sq}{\mathscr Q}
\newcommand{\mr}{\mathring}

\newcommand{\wh}{\widehat}
\newcommand{\lp}{L^{p}}

\newcommand{\ZR}{\mathbb{R}}
\newcommand{\ZT}{\mathbb{T}}
\newcommand{\ZZ}{\mathbb{Z}}

\newcommand{\ZC}{\mathbb{C}}
\newcommand{\ZB}{\mathbb{B}}
\newcommand{\ZI}{\mathbb{I}}

\newcommand{\ZM}{\mathbb{M}}

\newcommand{\ti}{\tilde}
\newcommand{\Id}{{\bf 1}}

\newcommand{\cT}{{\mathcal T}}

\newcommand{\bq}{{\bf q}}

\begin{document}

\title{On the Bochner-Riesz operator in $\ZR^3$}
\author{Shukun Wu}

\address{
Shukun Wu\\
Department of Mathematics\\
University of Illinois at Urbana-Champaign\\
Urbana, IL, 61801, USA}

\email{shukunw2@illinois.edu}
\date{}

\begin{abstract}
We improve the Bochner-Riesz conjecture in $\ZR^3$ to $\max\{p,p/(p-1)\}\geq3.25$. Our main methods are the Bourgain-Guth broad-narrow argument in \cite{Bourgain-Guth-Oscillatory}, and the iterated polynomial partitioning used in \cite{Guth-restriction-R3} (implicitly) and \cite{Wang-restriction-R3}. The main novelty of this paper is a backward algorithm that emerges from the iterated polynomial partitioning we used. This algorithm helps us realize a geometric observation from the tangential contributions.
\end{abstract}
\maketitle

\section{Introduction}
\setcounter{equation}0
The purpose of this paper is to present a small improvement of the Bochner-Riesz conjecture in $\ZR^3$. Recall that for $\la>0$, the Bochner-Riesz multiplier of order $\la$ in $\ZR^n$ is defined by  
\begin{equation}
\label{B-R-operator}
    T^\la f(x)=\int_{\ZR^n}(1-|\xi|^2)^\la_+\wh{f}(\xi)e^{ix\cdot\xi}d\xi.
\end{equation}

\begin{conjecture}(Bochner-Riesz)
\label{Bochner-Riesz}
Assume $f\in L^p(\ZR^n)$ and $1\leq p\leq \infty$. Then
\begin{equation}
\label{B-R-conj}
    \|T^\la f\|_p\leq C_{p,\la}\|f\|_p
\end{equation}
for $\la>\la_{n,p}$, where the factor $\la_{n,p}$ is defined to be
\begin{equation}
\label{lambda-p}
    \la_{n,p}=\max\Big\{0,n\Big|\frac{1}{p}-\frac{1}{2}\Big|-\frac{1}{2}\Big\}.
\end{equation}
\end{conjecture}

The Bochner-Riesz operator was introduced by Bochner in the 1930s, aiming to understand the radical convergence of Fourier transform. Since then, the Bochner-Riesz conjecture plays a crucial role in Fourier analysis. In the early 1970s, Carleson and Sj\"olin \cite{Carleson-Sjolin}, as well as Fefferman \cite{Fefferman-spherical}, settled the Bochner-Riesz conjecture in the plane. While for $\ZR^n$, $n\geq3$, the conjecture is widely open.

In higher dimensions, Tomas \cite{Tomas} proved that the Bochner-Riesz conjecture is true when $\max\{p,p/(p-1)\}\geq 2(n+2)/(n-1)$, via a $TT^\ast$ method. This result was improved by Bourgain \cite{Bourgain-Besicovitch} later in 1991, using new estimates for the Nikodym maximal function. After that, improvements have been made by several authors. See for instance,  \cite{Wolff-Kakeya}, \cite{Lee}, \cite{Bourgain-Guth-Oscillatory}. To the author's knowledge, in $\ZR^3$, the best result so far is due to Lee \cite{Lee} and \cite{Lee-Hormander}, who proved that the Bochner-Riesz conjecture is true when $\max\{p,p/(p-1)\}\geq 10/3$; In $\ZR^n$, $n\geq4$, the best results are given by Guth, Hickman and Iliopoulou \cite{G-H-I}. They in fact proved a more general result with H\"ormander-type operators, which automatically implies the Bochner-Riesz conjecture at the endpoints
\begin{eqnarray*}
&&\max\{p,p/(p-1)\}\geq2\frac{3n+1}{3n-3}\hspace{5mm}{\rm if}~n~{\rm is~odd,}\\[1ex]
&&\max\{p,p/(p-1)\}\geq2\frac{3n+2}{3n-2}\hspace{5mm}{\rm if}~n~{\rm is~even}.
\end{eqnarray*}

\vspace{3mm}

The main result of this paper is the following improvement of the Bochner-Riesz conjecture in $\ZR^3$. 
\begin{theorem}
\label{B-R-theorem-sharp}
Let $T^\la$ be the Bochner-Riesz operator defined in \eqref{B-R-operator}. Then the Bochner-Riesz conjecture \eqref{B-R-conj} holds when $n=3$, $\max\{p,p/(p-1)\}\geq 3.25$.
\end{theorem}
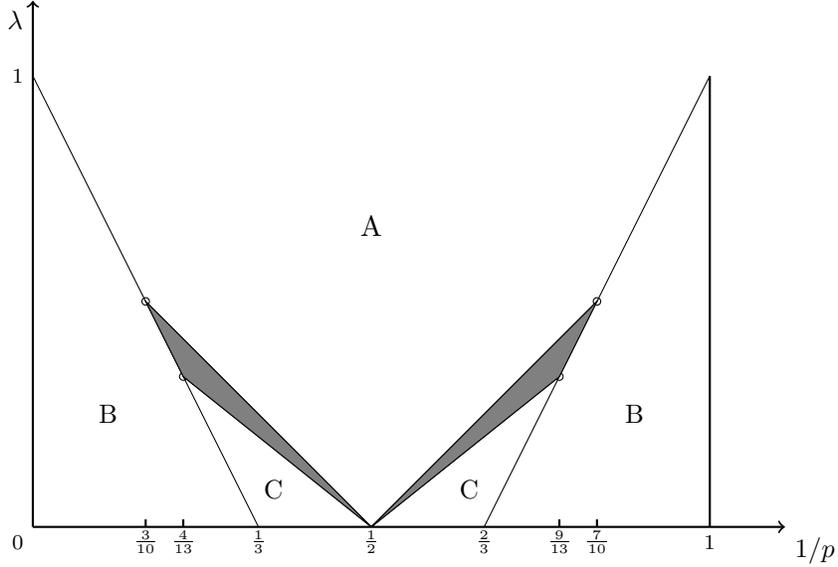
\begin{figure}
\begin{tikzpicture}
\draw[->,thick] (0,0) -- (0,7) node [anchor=north east] {$\la$};
\draw[->,thick] (0,0) -- (10,0) node [anchor=north west] {$1/p$};
\draw[thick] (9,0) -- (9,6) node [anchor=north east] {};

\node at (4.5,-0.2) {\tiny $\frac{1}{2}$};

\node at (3,-0.2) {\tiny $\frac{1}{3}$};

\node at (6,-0.2) {\tiny $\frac{2}{3}$};

\node at (-0.2,6) {\footnotesize $1$};
\node at (9,-0.2) {\footnotesize $1$};
\node at (-0.2,-0.2) {\footnotesize 0} ;

\node at (2,-0.2) {\tiny $\frac{4}{13}$};
\draw[thick] (2,0) -- (2,0.1);
\node at (1.5,-0.2) {\tiny $\frac{3}{10}$};
\draw[thick] (1.5,0) -- (1.5,0.1);
\node at (7,-0.2) {\tiny $\frac{9}{13}$};
\draw[thick] (7,0) -- (7,0.1);
\node at (7.5,-0.2) {\tiny $\frac{7}{10}$};
\draw[thick] (7.5,0) -- (7.5,0.1);
\draw (0,6) -- (3,0);
\draw (9,6) -- (6,0);
\draw (2,2) circle (.5mm);
\draw (1.5,3) circle (.5mm);
\draw (7,2) circle (.5mm);
\draw (7.5,3) circle (.5mm);
\draw (2,2) -- (4.5,0);
\draw (1.5,3) -- (4.5,0);
\draw (7,2) -- (4.5,0);
\draw (7.5,3) -- (4.5,0);
\draw [fill=gray] (4.5,0) to (2,2) to (1.5,3) to  (4.5,0) ;
\draw [fill=gray] (4.5,0) to (7,2) to (7.5,3) to  (4.5,0) ;
\node at (8,1.5) {B};
\node at (1,1.5) {B};
\node at (4.5,4) {\large A};
\node at (3.2,0.5) {C};
\node at (5.8,0.5) {C};

\end{tikzpicture}

\caption{$L^p$ behavior for the Bochner-Riesz operator \eqref{B-R-operator} in $\ZR^3$.    } 
\label{figure-graph}

\end{figure}

We let $K^\la$ be the kernel of the Bochner-Riesz operator, namely, $T^\la f=K^\la\ast f$. Recall that $K^\la(x)$ has an asymptotic expansion
\begin{equation}
    K^\lambda(x)\approx|x|^{-(n+1)/2-\lambda}\left(e^{i|x|}\sum_{j=0}^{\infty}a_j|x|^{-j}+e^{-i|x|}\sum_{j=0}^\infty b_j|x|^{-j}\right)
\end{equation}
for some constants $a_j,b_j$ as $x\to\infty$. This shows that the estimates \eqref{B-R-conj} fail when $\la\leq\la_{n,p}$, by taking $f$ to be a smooth test function whose Fourier transform equals to 1 in the unit ball. In particular, when $n=3$, the estimates \eqref{B-R-conj} fail in the region $B$ in Figure \ref{figure-graph}. The region $A$ was obtained by Lee in \cite{Lee}, while the region $C$ remains open. The shaded area shows the improvement we will make in this paper.

\vspace{3mm}

The main method we use to prove Theorem \ref{B-R-theorem-sharp} is polynomial partitioning, which was introduced by Guth and Katz \cite{Guth-Katz}. In \cite{Guth-restriction-R3}, Guth applied this idea to restriction estimates in Fourier analysis. A nice overview of polynomial partitioning can be found in \cite{Guth-restriction-R3}. Broadly speaking, given a finite measure $\mu$ in $\ZR^n$, one can use the Borsuk-Ulam theorem to find a polynomial of degree $d$, such that the zero set of this polynomial partitions $\ZR^n$ into $\sim d^n$ many components, each of which has the same $\mu$-measure. These components are often called ``cells". It should be emphasized that the lower bound of the number of cells is crucial in the polynomial partitioning method. That is, we need the number of cells to be greater than $c d^n$, for a small absolute constant $c$.

The polynomial partitioning method also uses some basic theorems for varieties. For example: The Fundamental Theorem of Algebra, the B\'ezout's theorem. We will slightly modify Guth's original argument in \cite{Guth-restriction-R3} for the polynomial partitioning method. See Section 4 for details.

In the landmark paper \cite{Guth-restriction-R3}, Guth used polynomial partitioning to improve the restriction conjecture in $\ZR^3$ to $p>3.25$. This result was later strengthened to $p>13/3$ by Wang \cite{Wang-restriction-R3}. Specifically, Wang observed that one would have some extra geometric structures among wave packets by using polynomial partitioning iteratively. Thus, utilizing the idea of Wolff's two-ends argument in \cite{Wolff-bilinear}, Wang introduced a relation between wave packets and large balls, similar to the one in \cite{Tao-bilinear}, to obtain the best result for restriction conjecture in $\ZR^3$ so far. Howerer, we are unable to use Wang's idea in this paper. Instead, motivated by Wang's work, we seek for connections between wave packets at different scales. We also use the dyadic pigeonholing trick to realize our geometric observations. Details are presented in Section 7.

One difficulty for attacking the Bochner-Riesz conjecture is that, there is not an efficient way to pass $L^2$ space back to $L^p$ space. In this paper, we will use a combination of square function, Littlwood-Paley theorem for translated cubes, and the Nikodym maximal function to help us move back to $L^p$ space. The idea of using the Nikodym maximal function for studying the Bochner-Riesz conjecture dates back to \cite{Cordoba-B-R}. See also \cite{Bourgain-Besicovitch}. We point out that we have no loss when estimating the Nikodym maximal function in this paper, even below the critical index $q=3$ for the Nikodym maximal conjecture.

It is known that (See for instance \cite{Carbery}, \cite{Tao-B-R-restriction}) the restriction conjecture and the Bochner-Riesz conjecture are closely related. As a result, one would expect that progress in one problem would impact the other. Recent breakthroughs in the restriction conjecture encourage us to work on the Bochner-Riesz conjecture.

\vspace{3mm}

For convenience, we will prove the following version of Theorem \ref{B-R-theorem-sharp}. The case $p<2$ can be obtained by duality.
\begin{theorem}
\label{B-R-theorem-sharp-2}
Let $T^\la$ be the Bochner-Riesz operator defined in \eqref{B-R-operator}. Then when $n=3$, for any Schwartz function $f$, any $p\geq3.25$,  $\e<100$, we have
\begin{equation}
\label{B-R-estimate2}
    \|T^{\la_{3,p}+3\e} f\|_p\leq C_{p,\e}\|f\|_p,
\end{equation}
where $\la_{3,p}$ was defined in \eqref{lambda-p}.
\end{theorem}
The assumption $\e<100$ is just a minor technical issue. Our method is robust for $\e\geq1/100$. From now on, we fix $\e<1/100$, and let $\be=\e^{1000}$, $d=\e^6$, $\de=\e^2$ for convenience. We also let $E_\e=\e^{-\e^{-10000}}$ and assume $3<p<10/3$.

\vspace{3mm}

This paper is organized as follows: In Section 2, we will include some basic techniques in harmonic analysis, and a standard decomposition of the Bochner-Riesz operator. Section 3 is devoted to a revision of wave packet decomposition. In Section 4, we will modify Guth's original polynomial partitioning, and use the modification repeatedly in Section 5 to build up our iterated polynomial partitioning algorithm. After that, we will focus on the tangential contribution in Section 6. Finally, we will create a backward algorithm to pile things up in Section 7, and conclude the proof of our main result.

\vspace{3mm}

{\bf Notations: }
Throughout the paper, we will use the following notations:
\begin{enumerate}
    \item[$\bullet$] We let $a\sim b$ mean that $ca\leq b\leq Ca$ for some unimportant constants $c$ and $C$. We also use $a\lesssim b$ to represent $a\leq Cb$ for an unimportant constant $C$, and use $a\lesssim_\e b$ to reprensent $a\leq C_\e b$ for a big constant $C_\e$ depends only on $\e$. We remark that these constants may change from line to line.
    \item[$\bullet$]We use $B^n(x,r)$ to represent the open ball centered at $x$, of radius $r$, in $\ZR^n$, and use $B^n_r$ to represent the ball $B^n(0,r)$. In particular, we will use $B_r$ to denote the ball $B^3_r$. For any point $x=(x_1,\ldots,x_n)\in\ZR^n$, we let $\bar{x}=(x_1,\ldots,x_{n-1})$ be the first $n-1$ coordinates of $x$.
    \item[$\bullet$] Assuming $T$ is a finite rectangular tube in $\ZR^n$, we let $c(T)\in\ZR^n$ be the center of $T$, and let $\{e_1,\ldots,e_n\}$ be an orthonormal frame associated to $T$. That is, $e_j$ is a vector parallels to the $j$-th side of $T$, $1\leq j\leq n$. For any point $x\in\ZR^n$, we define a dilation of $x$ with respect to $T$ by $x/T:=(x_1/l_1,\ldots,x_n/l_n)$, where $l_1,\ldots,l_n$ are the side lengths of $T$ and $(x_1,\ldots,x_n)$ is the coordinate of $x$ under the frame $\{e_1,\ldots,e_n\}$.
    \item[$\bullet$] For any finite rectangular tube $T\subset\ZR^n$, we define a weight function $w_{N,T}(x)$ associated to it by  $w_{N,T}(x)=(1+|(x-c(T))/T|)^{-N}$.
    \item[$\bullet$] For any function $f$ defined in $\ZR^n$, we use $Z(f)$ to denote the zero sets of $f$, $Z(f)=\{x\in\ZR^n:f(x)=0\}$, and use $Z(f_1,\ldots,f_k)$ to denote the set $Z(f_1)\cap\cdots\cap Z(f_k)$. We also use both $N_rX$ and $N_r(X)$ to denote the $r$-neighborhood of the set $X$, for any $X\subset\ZR^n$ and a positive number $r$.
\end{enumerate}

\vspace{3mm}

\noindent
{\bf Acknowlegement}. I would like to thank my advisor Xiaochun Li for his encouragements throughout the project. I would also like to thank Mengzhudong Feng and Jiahao Hu for helpful discussions related to algebraic geometry.

\section{Preliminaries and basic setups}
In this section, we will first review some basic techniques in harmonic analysis. Then, we will decompose the Bochner-Riesz operator \eqref{B-R-operator} and work on a specific model operator, which we will carefully study in Section 5. Since such decomposition works for all dimensions, we prefer not to restrict ourselves in $\ZR^3$ here.

\vspace{3mm}
\subsection{Smooth mollifier} \hfill

We construct a smooth mollifier that will help us built up smooth partitions of unity and smooth cutoff functions throughout the paper. Of course there are various kinds of mollifiers, while here we just give one typical example.

We first build up a smooth mollifier in $\ZR$. Let $a(x)$ be the function
\begin{equation}
    a(x)=\Bigg(\prod_{k=1}^{E_\e}\frac{\sin (2^{-k}x)}{2^{-k}x}\Bigg)^2.
\end{equation}
Since
\begin{equation}
    \int_{\ZR}2^{k-1}\chi_{[-2^{-k},2^{-k}]}(\xi)e^{ix\cdot\xi}d\xi=\frac{\sin 2^{-k}x}{2^{-k}x},
\end{equation}
we can check directly that
\begin{enumerate}
    \item $1\geq a(x)\geq c$ when $x\in[-1,1]$;
    \item $0\leq \wh{a}(\xi)\leq1$, $\wh{a}(\xi)$ is supported in $[-2,2]$ and $\wh{a}(\xi)\geq c$ on $[-1,1]$;
    \item $a(x)\leq 2^{N^N}|x|^{-N}$ for any $|x|\geq 1$, where $N=E_\e$.
    \item As a consequence of $(1)$ and $(3)$, $a(x)\leq N^{N^N}(1+|x|)^{-N}$.
    \item We have $\int a(x)dx, \int \wh{a}(\xi)d\xi<\infty$.
\end{enumerate}

\vspace{3mm}

Based on the function $a(x)$ defined in $\ZR$, we can construct a smooth mollifier in $\ZR^n$. It is standard to use the smooth mollifiers to construct different kinds of smooth cutoff functions with respect to any rectangular parallelepiped in $\ZR^n$, such that these cutoff functions have good support conditions, regularity conditions, and decay conditions. We can also use the smooth mollifier to construct different smooth partitions of unity with respect to any cover of $\ZR^n$ using congruent rectangular parallelepipeds.

We will use these constructions for free throughout the paper. Since all the constructions stated here are quite standard, we omit the details. 

\subsection{A local $L^2$ estimate}\hfill

Let $\Ga$, $\Ga=\{(\bar\xi,\Phi(\bar\xi)):\bar\xi\in\overline{B^{n-1}(0,1)}\}$, be the truncated graph of a function $\Phi:\ZR^{n-1}\to\ZR$ with bounded Hessian. Next, we state a local $L^2$ estimate for functions whose Fourier transforms are supported in a small neighborhood of $\Ga$. A similar result for $\Ga$ being the truncated cone was stated in \cite{Wolff-bilinear}.
\begin{lemma}
Let $f$ be an $L^2$ function such that $\wh{f}$ is supported in $N_{\rho}(\Ga)$,  where $\rho$ is a positive number much smaller than 1. Let $1\geq\si\geq\rho$ and let $B\subset\ZR^n$ be a ball of radius $\si^{-1}$ in the time space. We assume that $\vp_B$ is a smooth cutoff function with respect to $B$ that $|\vp_B(x)|\geq c$ on $B$, $\wh\vp_B$ is supported in $B(0,3\si)$ and $|\wh\vp_B(\xi)|\geq c$ on $B(0,\si)$. Then
\begin{equation}
\label{local-l2}
    \|f\vp_B\|_2\lesssim\rho^{1/2}\si^{-1/2}\|f\|_2.
\end{equation}
\end{lemma}
\begin{proof}

We partition $\overline{B^{n-1}(0,1)}$ into finitely overlapping balls $\{Q\}$ of radius $\rho^{1/2}$. Since $|\nabla^2\Phi|$ is bounded, the set $\Ga_Q=\{(\bar\xi,\Phi(\bar\xi)):\bar\xi\in Q\}$ is contained in a rectangular tube $\om_Q$ of dimensions $\sim\rho^{1/2}\times\cdots\times\rho^{1/2}\times\rho$, and hence $N_{\rho}(\Ga_Q)\subset 2\om_Q$. Notice that the sets $\{\om_Q\}$ are finitely overlapped. Thus, for any $1\geq\si\geq\rho$, any $\si$ ball $q\subset\ZR^n$, we have
\begin{equation}
    |q\cap N_{\rho}(\Ga)|\lesssim \rho\si^{-1}|q|.
\end{equation}

By Plancherel, $\|f\vp_B\|_2=\|\wh{f}\ast\wh\vp_B\|_2$. We square $\|\wh{f}\ast\wh\vp_B\|_2$ and recall that $\wh{f}$ is supported in $N_\rho(\Ga)$, so
\begin{equation}
    \|\wh{f}\ast\wh\vp_B\|_2^2=\int \Id_{N_\rho(\Ga)}(\xi)\Id_{N_\rho(\Ga)}(\xi')\wh{f}(\xi)\wh\vp_B(\eta-\xi)\overline{\wh{f}(\xi')\wh\vp_B(\eta-\xi')}d\eta d\xi d\xi'.
\end{equation}
Let $K(\xi,\xi')$ be the kernel
\begin{equation}
    K(\xi,\xi')=\int\wh\vp_B(\eta-\xi)\overline{\wh\vp_B(\eta-\xi')}\Id_{N_\rho(\Ga)}(\xi)\Id_{N_\rho(\Ga)}(\xi')d\eta.
\end{equation}
It is straightforward to check that
\begin{equation}
    \int |K(\xi,\xi')|d\xi,~~\int |K(\xi,\xi')|d\xi'\lesssim \rho\si^{-1}.
\end{equation}
Therefore, by Cauchy-Schwarz inequality and Schur's test, $\|\wh{f}\ast\wh\vp_B\|_2^2\lesssim \rho\si^{-1}\|\wh{f}\|_2^2$. Take square root for both sides to conclude our lemma.
\end{proof}
We will not use this lemma until Section 5. While we choose to state it here because the lemma is elementary.

\subsection{Littwood-Paley theorem for translated cubes}\hfill

There is a lot of literature including the study of Littwood-Paley theorem for translated cubes or rectangles. For example, \cite{Cordoba-B-R}, \cite{Cordoba-L-P}, and a generalization in \cite{Rubio-de-Francia}. We choose to state the one in \cite{Cordoba-L-P} here, as it is what we need later in Section 7.
\begin{theorem}
\label{L-P-cubes}
Assume that $\{\om\}$ is a collection of congruent rectangular parallelepipeds in $\ZR^n$. Let $\{\wh\psi_\om\}$ be a collection of smooth cutoff functions associated to $\{\om\}$ such that $\wh\psi_\om$ is supported in $2\om$. Then for any Schwartz function $f$ and for any $2\leq p<\infty$, 
\begin{equation}
    \Big\|\Big(\sum_\om|\psi_\om\ast f|^2\Big)^{\frac{1}{2}}\Big\|_p\leq C_p\|f\|_p.
\end{equation}
\end{theorem}

\vspace{3mm}

The rest of this section is devoted to the decomposition of the Bochner-Riesz operator. There are two steps. In the first step, we will cut the Bochner-Riesz operator in the frequency space into pieces, so that each piece varies little in the angular direction. After that, we will focus on one piece, and use the Bourgain-Guth broad-narrow argument in \cite{Bourgain-Guth-Oscillatory} to decompose the piece into a broad part and a narrow part. The narrow part is easy to handle by induction, so we will set the broad part as our model operator.  

\subsection{Frequency decomposition for the Bochner-Riesz operator}\hfill

We partition the operator $T^\la f$ defined in \eqref{B-R-operator} in the following way: Let $\bq=\{q\}$ be the collection of lattice $1/4$ cubes in $\ZR^n$. We set $\{\psi_q\}$ to be a smooth partition of unity associated to ${\bf q}$ such that $\psi_q$ is supported in $2q$. Let ${\bf q}_1$ be the collection of cubes $q$ that $2q$ intersects $S^{n-1}$, and let ${\bf q}_2$ be the collection of cubes $q$ that $2q$ is contained in $B^{n}(0,1)$. For any cube $q$, define
\begin{equation}
    T_q^\la f(x)=\int_{\ZR^n}(1-|\xi|^2)^\la_+\psi_q(\xi)\wh{f}(\xi)e^{ix\cdot\xi}d\xi.
\end{equation}
We thus can partition $T^{\la}f$ as
\begin{equation}
\label{B-R-partition1}
    T^{\la}f=\sum_{q_1\in{\bf q}_1}T^{\la}_{q_1}f+\sum_{q_2\in{\bf q}_2}T^{\la}_{q_2}f.
\end{equation}
The second part of \eqref{B-R-partition1} is bounded by Hardy-Littlewood maximal function and hence it satisfies $L^p$ estimates. Since there are only finitely many $q$ in ${\bf q}_1$, it suffices to show for any $q\in{\bf q}_1$,  when $n=3$, $\la=\la_{3,p}+3\e$, $p\geq3.25$,
\begin{equation}
    \|T^{\la}_{q}f\|_p\leq C_{p,\e}\|f\|_p.
\end{equation}

Since the function $(1-|\xi|^2)^\la_+$ is radical, without loss of generality, we assume that $q$ intersects the axis $e_3$. In this case, the multiplier of $T^{\la}_q$ can be rewritten as
\begin{equation}
    (1-|\xi|^2)^{\la}_+\psi_q(\xi)=((1-|\bar{\xi}|^2)^{1/2}-\xi_n)^{\la}_+((1-|\bar{\xi}|^2)^{1/2}+\xi_n)^{\la}\psi_q(\xi).
\end{equation}
We absorb $((1-|\bar{\xi}|^2)^{1/2}+\xi_n)^{\la}$ into the smooth function $\psi_q(\xi)$ and partition $((1-|\bar{\xi}|^2)^{1/2}-\xi_n)_+^{\la}$ dyadically so that $T_q^\la f=\sum_{k\geq0}T_{q,k}^\la f$, with
\begin{equation}
    T_{q,k}^\la f(x)=2^{-\la k}\int_{\ZR^n}\phi_k\Big(\frac{\xi_n-(1-|\bar{\xi}|^2)^{1/2}}{2^{-k}}\Big)\psi_q(\xi)\wh{f}(\xi)e^{ix\cdot\xi}d\xi.
\end{equation}
Here $\phi_k$ is a smooth function supported in $[-2,0]$ and satisfies that $|\partial^\al\phi_k|\leq C_\e$ for all multi-indices $\al$ with $|\al|\leq E_\e$. By the triangle inequality, it suffices to show
\begin{equation}
\label{partitioned-estimate1}
    \|T_{q,k}^\la f(x)\|_p\leq C_{p,\e}2^{-\e k}\|f\|_p.
\end{equation}

Observe that when restricting $\bar\xi\in \overline{B^{n-1}(0,1/2})$, the Hessian of the function $-(1-|\bar{\xi}|^2)^{1/2}$ positive definite. This motives us to prove \eqref{partitioned-estimate1} with $(1-|\bar{\xi}|^2)^{1/2}$ replaced by a more general function.

\vspace{3mm}

Let $\Ga\subset B^n(0,2)$ be the graph of a $C^3$ function $\Phi:\overline{B^{n-1}(0,1/2)}\to\ZR$ with $\nabla\Phi(0)=\Phi(0)=0$ and that $-\nabla^2\Phi$ is positive definite. We will next define a collection of multipliers related to the hypersurface $\Ga$. Let $\psi_1(\bar\xi),\psi_2(\xi_n)$ be two smooth functions satisfying the following conditions:
\begin{enumerate}
    \item The support of $\psi_1$ is contained in the ball $B^{n-1}(0,1/2)$, and the support of $\psi_2$ is contained in the interval $[-2,0]$.
    \item $|\wh\psi_1(\bar x)|\lesssim_\e w_{E_\e,B^{n-1}_1}(\bar x)$ and $|\wh\psi_2(x_n)|\lesssim_\e w_{E_\e,B_1^1}(x_n)$.
    \item $|\partial^{\al_1}\psi_1|\lesssim C_\e$ and $|\partial^{\al_2}\psi_2|\lesssim C_\e$, when $|\al_1|,\al_2\leq E_\e$.
\end{enumerate}
Define the multiplier $m(x)$ by
\begin{equation}
\label{multipler-def}
    \wh{m}(\xi)=\psi_1(\bar\xi)\psi_2(R(\xi_n-\Phi(\bar\xi))).
\end{equation}
We denote by $\ZM(R)$ all the multipliers $m$ that have the formulation \eqref{multipler-def}.

Now fix a multiplier $m\in\ZM(R)$, we define a new operator $Sf$ as
\begin{equation}
\label{multiplier-definition}
    Sf(x)=\int_{\ZR^n} e^{ix\cdot \xi}\wh{m}(\xi)\wh{f}(\xi)d\xi.
\end{equation}
Our strongest result is the following theorem:
\begin{theorem}
\label{multiplier-theorem}
Set $n=3$. Let $Sf$ be defined as above and assume $m\in\ZM(R)$. Then for $p\geq3.25$,
\begin{equation}
\label{multiplier-estimate}
    \|Sf\|_p\leq C_\e R^{\la_{3,p}+2\e}\|f\|_p.
\end{equation}
\end{theorem}

We can deduce \eqref{partitioned-estimate1} from Theorem \ref{multiplier-theorem} immediately by introducing a new partition of unity for the variable $\bar\xi$, and absorbing $\psi_q(\xi)$ in the function $\wh{f}(\xi)$. Therefore, we get \eqref{B-R-estimate2}. \qed

\subsection{The Broad-Narrow argument}\hfill

It remains to show Theorem \ref{multiplier-theorem}. By partitioning our target operator $Sf$ defined in \eqref{multiplier-definition} finer, and combining some rescaling arguments, we can assume that $\Phi(0)=\nabla\Phi(0)=0$, $\nabla^2\Phi(0)=-I_2$, and $|\nabla^3\Phi|\leq 10^{-10}$.

We let $K=C_0 R^{\e^{10}}$ be a large number, where $C_0$ is a big constant we will choose later from induction. Let $\cq^{\bf tau}=\{Q_\tau\}$ be the collection of vertical lattice $K^{-1}$ tubes of infinite length. Namely, $Q_\tau$ has cross-section diameter $\sim K^{-1}$ and length infinity, and the long side of $Q_\tau$ is parallel to the axis $e_3$. We set $\{\psi_{Q_\tau}\}$ be a smooth partition of unity associated to $\cq^{\rm tau}$ such that $\psi_{Q_\tau}$ is supported in $2Q$, $|\partial^\al\psi_{Q_\tau}|\leq C_\e K^{|\al|}$ for $|\al|\leq E_\e$. Hence, by setting $\wh{f}_\tau=\wh{m}_\tau\wh{f}:=\wh{m}\psi_{Q_\tau}\wh{f}$, we have a frequency decomposition of $Sf$ that
\begin{equation}
    Sf=\sum_{\tau}f_\tau.
\end{equation}

Since $\nabla^2\Phi(0)=-I_2$ and $|\nabla^3\Phi|\leq 10^{-10}$, the Fourier support of the function $f_\tau$ in contained $2\tau$, where $\tau$ is a rectangular tube of dimensions $\sim K^{-1}\times K^{-1}\times K^{-2}$. Here $2\tau$ is the rectangle that has the same center as $\tau$, but twice the side length of $\tau$. We let $\cT=\{\tau\}$ be the collection of rectangular tubes that $f_\tau$ is not 0.  We further decompose $\cq^{\bf tau}$ into 100 subsets $\cq_1^{\bf tau},\ldots,\cq_{100}^{\bf tau}$, such that any two vertical tubes $Q_{\tau_1},Q_{\tau_2}\in\cq_j^{\bf tau},j=1,\ldots,100$, have distance $\geq 9K^{-1}$. Thus,
\begin{equation}
\nonumber
    Sf=\sum_{j=1}^{100}S_j f=\sum_{j=1}^{100}\sum_{Q\in\cq_j^{\bf tau}}\int_{\ZR^3} e^{ix\cdot\xi}(\wh{m}\psi_Q)(\xi)\wh{f}(\xi)d\xi=:\sum_{j=1}^{100}\int_{\ZR^3} e^{ix\cdot\xi}\wh{m}_j(\xi)\wh{f}(\xi)d\xi.
\end{equation}
By the triangle inequality, up to a constant loss, it suffices to assume $m=m_j$ and $Sf=S_jf$ for a particular $j$.

\vspace{3mm}

Next, we let $\cq^{\bf theta}=\{Q_\theta\}$ be the collection vertical lattice $R^{-1/2}$ tubes of infinite length. We similarly define $\{\psi_{Q_\theta}(\xi)\}$ to be smooth partition of unity associated to $\cq^{\bf theta}$, with good support and derivative conditions.

Since $|\nabla^2\Phi|\sim1$, the support of the multiplier $\wh{m}\psi_{Q_\theta}$ is contained in a rectangular tube $\theta$, whose dimensions roughly equal to $R^{-1/2}\times R^{-1/2}\times R^{-1}$, and whose shortest side is parallel to the vector $(\nabla\Phi(\overline{c(Q_\theta})),-1)$. We  write $\wh{\vp}_\theta=\wh{m}\psi_{Q_\theta}$ and
\begin{equation}
    f_\theta=\int_{\ZR^3} e^{ix\cdot\xi}\wh\vp_\theta(\xi)\wh{f}(\xi)d\xi
\end{equation}
in short. Thus, we have a finer frequency decomposition of $Sf$ that
\begin{equation}
    Sf=\sum_{Q_\theta\in\cq^{\bf theta}}\int_{\ZR^3} e^{ix\cdot\xi}(\wh{m}\psi_{Q_\theta})(\xi)\wh{f}(\xi)d\xi=\sum_\theta f_\theta.
\end{equation}
We say a rectangular tube $T$ of dimensions $R^{1/2}\times R^{1/2}\times R$ is  $dual$ to $\theta$, if its longer side is parallel to the vector $(\nabla\Phi(\bar{c_Q}),-1)$. Via a standard non-stationary phase method, we can conclude that if $T$ is the rectangular tube dual to $\theta$ centered at the origin, then
\begin{equation}
    |\vp_\theta(x)|\lesssim_\e w_{T,E_\e}(x).
\end{equation}

To conclude, we have two frequency decompositions for the operator $Sf$, 
\begin{equation}
\label{schwartz-fcn-theta}
    Sf=\sum_{\tau\in \cT} f_\tau=\sum_{\tau\in \cT}\sum_{\theta\in2\tau}f_\theta=\sum_\theta f_\theta.
\end{equation}
We denote by $\Theta$ the collection of all $\theta$ that appear in the last summation above.

\vspace{3mm}
Next, we will give the definition of $\al$-broadness.  Supposing $g$ is a function that $\wh{g}$ is supported in the union of $2\tau$, namely, $\cup_{\tau\in\cT}2\tau$. We let $\wh{g}_\tau=\Id_{\tau}\wh{g}$. Then, for a small positive number $\al$, we say a point $x\in\ZR^3$ is $\al$-broad with respect to $g$, if 
\begin{equation}
    \sup_{\tau\in\cT}|g_{2\tau}(x)|\leq \al|g(x)|.
\end{equation}

We let ${\rm Br}_\al g$ be $|g|$ if $x$ is $\al$-broad and $0$ otherwise. We aim to prove the following $\al$-broad estimates in the rest of the paper:
\begin{theorem}
\label{broad-theorem}
Let ${\rm Br}_\al Sf$ be defined as above and let $\al=K^{-\e}$. Assume that $f$ is supported in an $R$ ball in $\ZR^3$. Then for $p\geq3.25$, 
\begin{equation}
\label{broad-estimate}
    \|{\rm Br}_\al Sf\|_{\lp(B_R)}\leq C_\e R^{\la_{3,p}+\e}\|f\|_p.
\end{equation}
\end{theorem}
As we will see later, the broadness argument allows us to restrict the contribution of ${\rm Br}_\al Sf$ into a thin neighborhood of a collection of varieties, each of which has degree $\sim d=R^{\e^6}$.

\vspace{3mm}

At the end of this section, we will show how Theorem \ref{broad-theorem} implies Theorem \ref{multiplier-theorem}. We need the following localization lemma and will prove it latter.
\begin{lemma}
\label{local-lemma1}
Let $f$ be a Schwartz function defined in $\ZR^3$ and let $Sf$ be defined in \eqref{multiplier-definition}. Assume that for any $R$ ball $B$ in $\ZR^3$, 
\begin{equation}
\label{local-lemma-assumption}
    \|S(f\Id_{B})\|_{L^p(B_R)}\leq C_1\|f\Id_B\|_p.
\end{equation}
Then we have 
\begin{equation}
    \|Sf\|_p\leq  C_\be R^\be C_1\|f\|_p.
\end{equation}
\end{lemma}

\vspace{3mm}
Now we can show Theorem \ref{multiplier-theorem} via Theorem \ref{broad-theorem}. Our argument relies on the induction on scales method. We let $\la_p=\la_{3,p}$ for simplicity. By Lemma \ref{local-lemma1}, it suffices to show
\begin{equation}
\label{local-desired}
    \|S(f)\|_{L^p(B_R)}\leq C_\e R^{\la_p+2\e}(R^{-\be}C_\be^{-1})\|f\|_p,
\end{equation}
when $f$ is supported in an $R$ ball in $\ZR^3$.

Clearly, for any $x\in\ZR^3$,
\begin{equation}
\label{broad-narrow}
    |Sf(x)|\leq \al^{-1}\sup_{\tau\in\cT}|f_\tau(x)|+{\rm Br}_\al Sf(x).
\end{equation}
Take $p$-th power to both sides and use the fact $l^p\subset l^\infty$ so that
\begin{equation}
\label{broad-narrow-argument}
    \int_{B_R}|Sf|^p\leq 2^pK^{\e}\sum_{\tau\in\cT}\int_{B_R}|f_\tau|^p+2^p\int_{B_R}{\rm Br}Sf^p.
\end{equation}
Recall that
\begin{equation}
    f_\tau=\int_{\ZR^3} e^{ix\cdot\xi}\wh{m}_\tau(\xi)\wh{f}(\xi)d\xi.
\end{equation}
If we let $\vp_\tau$ be a smooth cutoff function of $2\tau$ that $\vp_\tau=1$ on $2\tau$, then $f_\tau=m_\tau\ast f_{\vp_\tau}$, where $f_{\vp_\tau}=\vp_\tau\ast f$. Notice that for each $\tau$, after parabolic rescaling
\begin{equation}
    \cl_\tau:(\bar{\xi},\xi_3)\to\big(\frac{\bar\xi-\overline{c(\tau)}}{K^{-1}},\frac{\xi_3-\nabla\Phi(\overline{c(\tau)})\cdot\bar\xi-\Phi(\overline{c(\tau)})+\nabla\Phi(\overline{c(\tau)})\cdot\overline{c(\tau)}}{K^{-2}}\big),
\end{equation}
the multiplier $m_\tau\circ \cl_\tau^{-1}\in\ZM(RK^{-2})$.

We use \eqref{multiplier-estimate} at the scale $RK^{-2}$ as an induction hypothesis and employ Lemma \ref{local-lemma1} again to obtain
\begin{equation}
    \int_{B_R}|f_\tau|^p\leq  C_\e^p R^{p\la_p+2p\e}K^{-2p\la_p-4p\e}(C_\be^p R^{p\be})\|f_{\vp_\tau}\|_p^p.
\end{equation}
Summing up all $\tau\in\cT$ and multiplying $2^pK^\e$ to both sides, we have
\begin{equation}
\label{induction-tau}
    2^pK^{\e}\sum_{\tau\in\cT}\int_{B_R}|f_\tau|^p\leq C_\e^pR^{p\la_p+2p\e}(2^pK^{-2p\la_p-4p\e}K^\e C_\be^p R^{p\be})\sum_{\tau\in\cT}\|f_{\vp_\tau}\|_p^p.
\end{equation}
Observe that $\|{f_{\vp_\tau}}\|_{L^rl^r}\leq C_r\|f\|_r$ is true for $r=2$ and $r=\infty$. We invoke the real interpolation so that
\begin{equation}
    \sum_{\tau\in\cT}\|f_{\vp_\tau}\|_p^p\leq C_p\|f\|_p^p.
\end{equation}
Plugging this back to \eqref{induction-tau} we get
\begin{equation}
    2^pK^{\e}\sum_{\tau}\int_{B_R}|f_\tau|^p\leq C_\e^pR^{p\la_p+2p\e}(2^pK^{-2p\la_p-4p\e}K^\e C_\be^p R^{p\be}C_p)\|f\|_p^p.
\end{equation}
Now we can take $C_0$ big enough in the definition $K=C_0 R^{\e^{10}}$ to ensure
\begin{equation}
\label{induction-narrow}
    2^pK^{\e}\sum_{\tau}\int_{B_R}|f_\tau|^p\leq C_\e^pR^{p\la_p+2p\e}2^{-p}(R^{-p\be}C_\be^{-p})\|f\|_p^p.
\end{equation}

On the other hand, Theorem \ref{broad-theorem} gives
\begin{equation}
\label{induction-broad}
    \int_{B_R}{\rm Br}_\al Sf^p\leq C_\e'^p R^{p\la_p+p\e}\|f\|_p^p\leq C_\e^p R^{p\la_p+2p\e}4^{-p}(R^{-p\be}C_\be^{-p})\|f\|_p^p.
\end{equation}
Combining \eqref{broad-narrow-argument}, \eqref{induction-narrow} and \eqref{induction-broad}, we prove \eqref{local-desired}. \qed

\vspace{3mm}
We are left with the proof of Lemma \ref{local-lemma1}.
\begin{proof}We know that $Sf=m\ast f$ for a kernel $m\in\ZM(R)$. By the method of non-stationary phase, we get that for $|x|\geq R$, $N=2000\e^{-2000}<E_\e$,
\begin{equation}
\label{kernel-decay}
    |m(x)|\leq C_N|R^{-1}x|^{-N}.
\end{equation}
Let $\cu,\cv$ be two collections of lattices $R$ cubes that cover $\ZR^3$. It suffices to show
\begin{equation}
    \sum_{U\in\cu}\|Sf\|_{L^p(U)}^p=\sum_{U\in\cu}\int\Id_U|\sum_{V\in\cv}m\ast(f\Id_V)|^p\leq C_\be R^{p\be}C_1\sum_{V\in\cv}\|f\Id_V\|_p^p.
\end{equation}

For fixed $U\in\cu$, we sort $V\in\cv$ according to ${\rm dist}(U,V)$. Let $\cv_k(U)=\{V:(R^{1+\e^{1500}})^{k-1}\leq{\rm dist}(U,V)\leq (R^{1+\e^{1500}})^k\}$ for $k\geq0$, and  let $(R^{1+\e^{1500}})^{-1}=0$ for convenience. Thus, by Minkowski's inequality, 
\begin{equation}
\nonumber
    \|Sf\|_{L^p(U)}^p=\int\Id_U|\sum_{V\in\cv}m\ast(f\Id_V)|^p\leq \Big(\sum_{k\geq0}\sum_{V\in\cv_k}\Big(\int\Id_U |m\ast(f\Id_V)|^p\Big)^{1/p}\Big)^p.
\end{equation}
Via a simple translation argument, we can use the hypothesis in Lemma \ref{local-lemma1} to bound those $V\in \cv_0$, and use \eqref{kernel-decay} to bound the other $V$, so that
\begin{equation}
    \|Sf\|_{L^p(U)}^p\leq\Big(\sum_{V\in\cv_0}C_1\|f\Id_V\|_p+\sum_{k\geq1}R^{-1000k}\sum_{V\in\cv_k}\|f\Id_V\|_p\Big)^p.
\end{equation}
Here we use the volume estimate $|U|\lesssim R^3$ and the Hausdorff-Young inequality. Since $|\cv_k|\leq R^{3k}(R^{\e^{1400}})^{3k+3}$, $p<4$, 
\begin{equation}
    \|Sf\|_{L^p(U)}^p\leq \Big(C_NC_1\sum_{k\geq0}R^{\e^{1300}}R^{-900k}\Big(\sum_{V\in \cv_k}\|f\Id_V\|_p^p\Big)^{1/p}\Big)^p,
\end{equation}
which is further bounded by 
\begin{equation}
    C_N^pC_1^pR^{\e^{1300}}\sum_{k\geq0}R^{-800k}\sum_{V\in \cv_k}\|f\Id_V\|_p^p.
\end{equation}
Summing up all $U\in\cu$ we finally have
\begin{equation}
    \|Sf\|_p^p\leq C_N'^pC_1^p R^{\e^{1200}}\sum_V\|f\Id_V\|_p^p=C_N'^pC_1^p R^{\e^{1200}}\|f\|_p^p.
\end{equation}
Noticing that $R^{\e^{1200}}\leq R^{p\be}$, we finish the proof of Lemma \ref{local-lemma1}.
\end{proof}

\section{Wave packet decomposition}
Wave packet decomposition is a standard tool in modern harmonic analysis. There are many elegant expositions on such decomposition, see for instance, \cite{Tao-bilinear} \cite{Guth-restriction-R3}. We will repeatedly use wave packet decomposition at different scales, so it is worthwhile to discuss this idea in a separate section.

\vspace{3mm}

Recall the equation \eqref{schwartz-fcn-theta} in Section 2
\begin{equation}
    Sf=\sum_{\theta\in\Theta} f_\theta.
\end{equation}

We start from the largest scale $R$.  For each $\theta$, we let $\ZT_\theta$ be the collection of lattice rectangular tubes dual to $\theta$. Let $\{\Id_T^\ast\}_{T\in\ZT_\theta}$ be a smooth partition of unity such that $\wh{\Id_T^\ast}$ is supported in $\theta$, $\Id_T^\ast\geq c$ on $3T$ and $\Id_T^\ast\leq C_\e w_{T,E_\e}$. As a result, we can write $f_\theta$ as a sum of wave packets 
\begin{equation}
    f_\theta=\sum_{T\in\ZT_\theta}f_\theta\Id^\ast_T:=\sum_{T\in\ZT_\theta}f_{\theta,T}.
\end{equation}

It follows that, for $1\leq p\leq100$ and any $\ZT_\theta'\subset\ZT_\theta$, we have
\begin{equation}
    \sum_{T\in\ZT_\theta'}|\Id_T^\ast|^p\leq C_\e,
\end{equation}
which implies
\begin{equation}
\label{wave-packet-orthogonality}
    \sum_{T\in\ZT_{\theta}'}\|f_{\theta,T}\|_p^p\leq C_\e\|f_\theta\|_p^p.
\end{equation}
Conversely, we have
\begin{equation}
\label{lp-almost-orthogonality}
    \big\|\sum_{T\in\ZT_\theta'}f_{\theta,T}\big\|_p^p\leq C_\be R^{\be}\sum_{T\in\ZT_\theta'}\|f_{\theta,T}\|_p^p+O\big(R^{-E_\e^{1/2}}\|f_\theta\|_2^p\big).
\end{equation}
This follows from a similar argument in the proof of Lemma \ref{local-lemma1}. When $p=2$, we can sum up all the $\theta\in\Theta$ and invoke Plancherel's theorem to have the almost $L^2$-orthogonality
\begin{equation}
\label{L2-orthogonality}
    \big\|\sum_{\theta\in\Theta}\sum_{T\in\ZT_\theta'}f_{\theta,T}\big\|_2^2\leq C_\be R^{\be}\sum_{\theta\in\Theta}\sum_{T\in\ZT_\theta'}\|f_{\theta,T}\|_2^2+O\big(R^{-E_\e^{1/2}}\|Sf\|_2^2\big).
\end{equation}

The Fourier support of each wave packet $f_{\theta,T}$ is a subset of $2\theta$. We let $\ZT=\cup_\theta\ZT_\theta$ be the union of all possible tubes. Sometimes we will write $f_T=f_{\theta,T}$ simply because the information of $\theta$ is implicitly contained in $T$. Also, for any tube $T\in\ZT$, we would write $T_\theta=T$ to indicate that $T$ is dual to $\theta$. We define the direction of a wave packet $f_{T}$ by the direction of the longer side of $T$. Heuristically, we can think of a wave packet $f_{\theta,T}$ as $\|f_{\theta,T}\|_\infty\chi_T$. 

We will repeatedly use \eqref{wave-packet-orthogonality} and \eqref{L2-orthogonality} in the rest of the paper. Since the loss $R^\be$ from Schwartz tails in our setting is negligible, we will use \eqref{L2-orthogonality} without mentioning the loss $R^\be$. At this stage, we finish the wave packet decomposition for the function $Sf$ at the scale $R$. 

\vspace{3mm}

For any intermediate scale $R^\e\leq r\leq R$, assuming with the same $\Theta$, we have another function 
\begin{equation}
    g=\sum_{\theta\in\Theta} g_\theta.
\end{equation}
We also assume that the Fourier support of each $g_\theta$ is contained in $N_{r^{-1}}(2\theta)$. We let $\cq^{\bf omega}=\{Q_\om\}$ be the collection of lattice $r^{-1/2}$ vertical rectangular tubes of infinite length. Then, as we seen before, those $\theta$ who fall into a $Q_\om$ are automatically contained in a rectangular tube $\om$ of dimensions roughly equal to $r^{-1/2}\times r^{-1/2}\times r^{-1}$. We define $g_\om$ to be the sum of functions $g_\theta$ whose Fourier supports are contained in $\om$, and possibly some extra functions $g_\theta$ whose Fourier supports intersect the boundary of $\om$, as we require
\begin{equation}
    g=\sum_{\om}g_\om.
\end{equation}
We will use $\theta\sim\om$ to indicate that $g_\theta$ is added in the function $g_\om$. The exact formula for $g_\om$ is not important to us. We will only use the properties that $g_\om$ is a sum of $\sim Rr^{-1}$ many functions $g_\theta$, and that for any two rectangular tubes $\theta,\om$ with $\theta\sim\om$, the directions of the shortest sides of both $\theta$ and $\om$ make an angle $\lesssim r^{-1/2}$.

For each $\om$, we let $\ZI_\om$ be the collection of lattice rectangular tubes dual to $\om$. Let $\{\Id_I^\ast\}_{I\in\ZI_\om}$ be a smooth partition of unity such that $\wh{\Id_I^\ast}$ is supported in $\om$, $|\Id_I^\ast|\leq C_\e w_{I,E_\e}$ and $|\Id_I^\ast|\geq c$ on $3I$. Consequently, for $1\leq p\leq100$,
\begin{equation}
    \sum_I|\Id_I^\ast|^p\leq C_\e.
\end{equation}
We write
\begin{equation}
    g_\om=\sum_{I\in\ZI_\om}g_\om\Id^\ast_I:=\sum_{I\in\ZI_\om}g_{\om,I}.
\end{equation}
The Fourier support of each $g_{\om,I}$ is contained in $2\om$. Similarly, for any $\ZI_\om'\subset\ZI_\om$, 
\begin{equation}
    \sum_{I\in\ZI_{\om}'}\|g_{\om,I}\|_p^p\leq C_\e\|g_\om\|_p^p,
\end{equation}
\begin{equation}
    \big\|\sum_{I\in\ZI_\om'}g_{\om,I}\big\|_p^p\leq C_\be r^{\be}\sum_{I\in\ZI_\om'}\|g_{\om,I}\|_p^p+O\big(r^{-E_\e^{1/2}}\|g_\om\|_2^p\big),
\end{equation}
and
\begin{equation}
\label{L2-orthogonality2}
    \big\|\sum_{\om}\sum_{I\in\ZI_\om'}g_{\om,I}\big\|_2^2\leq C_\be r^{\be}\sum_{\om}\sum_{I\in\ZI_\om'}\|g_{\om,I}\|_2^2+O\big(r^{-E_\e^{1/2}}\|g\|_2^2\big).
\end{equation}
Since $r^{-E_\e^{1/2}}\leq R^{-\e E_\e^{1/2}}\leq R^{-E_\e^{1/4}}$, the error term is always negligible. For convenience, if $\|g\|_2\lesssim_\e R^{10}\|Sf\|_2$, we will similarly drop the error term automatically afterwards. We let $\ZI=\cup_\om\ZI_\om$  and finish the wave packet decomposition for the function $g$ at the scale $r$.

\section{More on polynomial partitioning}
In this section, we will state and prove a modification of Guth's original polynomial partitioning argument in \cite{Guth-restriction-R3}, and use this modified polynomial partitioning in the rest of the paper.

First, we make some definitions regarding to the zero sets of polynomials.
\begin{definition}
Suppose $Q_1,\ldots,Q_k$ are polynomials in $\ZR^n$. We say $Z(Q_1,\ldots,Q_k)$ is a transverse complete intersection if for any $x\in Z(Q_1,\ldots,Q_k)$, the vectors $\nabla Q_1(x),\ldots,\nabla Q_k(x)$ are linearly independent.
\end{definition}
\begin{definition}
We say a polynomial $P$ in $\ZR^n$ is non-singular, if $\nabla P(x)\not=0$ for any $x\in Z(P)$.
\end{definition}

The main result of this section is the following proposition.
\begin{proposition}
\label{final-partition}
Let $g$ be a non-negative $L^1$ function in $\ZR^n$, supported in the ball $B^n(0,R)$. Then for any $d\in\ZZ^+$, there exists a polynomial $P$ with degree $O(d)$, such that
\begin{enumerate}
    \item There are $\sim d^n$ many cells $O$ contained in $\ZR^n\setminus Z(P)$, satisfying 
    \begin{equation}
        \int_{O} g\sim d^{-n}\int_{\ZR^n}g\,.
    \end{equation}
    \item Each of these cells $O$ in {\rm (1)} lies in a cube in $B^n(0,2R)$ of diameter $\sim Rd^{-1}$.
\end{enumerate}
\end{proposition}
This proposition was proved by Wang \cite{Wang-restriction-R3}, via the Milnor-Thom theorem. Here we will give another proof of the proposition, by modifying Guth's original argument for polynomial partitioning, so that we will not need to use Milnor and Thom's result. Our argument relies on a generalization of B\'ezout's theorem.

\vspace{3mm}
The rest of this section is devoted to the proof of Proposition \ref{final-partition}. We begin with some preparations. Recall the polynomial ham sandwich theorem proved in \cite{Guth-restriction-R3}.
\begin{theorem}
\label{polynomial-sandwich}
Let $g_1,\ldots,g_N\in L^1(\ZR^n)$. Then there exist a polynomial $P$ of degree $\leq C_n N^{1/n}$ such that for each $g_j$,
\begin{equation}
    \int_{\{P>0\}}g_j=\int_{\{P<0\}}g_j\,.
\end{equation}
\end{theorem}
Let ${\rm Poly}_d(\ZR^n)$ be the vector space of polynomials in $\ZR^n$ with degree at most $d$. Guth showed in \cite{Guth-restriction-R3} that non-singular polynomials with degree at most $d$ are dense in ${\rm Poly}_d(\ZR^n)$. As a generalization, we have the following lemma
\begin{lemma}
\label{dense-lemma}
Let $M$ be a smooth manifold in $\ZR^n$ with dimension at least 1. Assuming $V$ is an affine subspace of ${\rm Poly}_d(\ZR^n)$ such that $V+c=V$ for any $c\in\ZR$. Let $V'\subset V$ be the set of polynomials that $\nabla_MP(x)\not=0$ for any $P\in V'$ and any $x\in M\cap Z(P)$. Here $\nabla_M$ is the gradient operator on the smooth manifold $M$. Then $V'$ is a dense subset in $V$. Also, the complement of $V'$ with respect to $V$ has measure 0.
\end{lemma}
\begin{proof}
We consider the map $\cp:M\times V\to \ZR\times V$, $(x,P)\mapsto (P(x),P)$. Since $M$ is a smooth manifold, the map $\cp$ is smooth. By Sard's theorem, regular values of $\cp$ are dense in $\ZR\times V$. Notice that $(c,P)$ is a regular value of $\cp$ if and only if $P(x)=c$ and $\nabla_MP(x)\not=0$. By Fubini's theorem, there is a number $c\in\ZR$ such that for almost every polynomial $P\in V$, $(c,P)$ is a regular value of $\cp$. This implies $\nabla_M(P(x)-c)\not=0$ on the set $\{x\in M:P(x)-c=0\}$, for almost every polynomial $P\in V$. Since the polynomial $P-c$ also belongs to $V$, we can conclude that $V'$ is dense and its complement with respect to $V$ has measure 0. 
\end{proof}
When $M$ is a point, similar to Lemma \ref{dense-lemma}, we have
\begin{lemma}
\label{dense-lemma-2}
Assume that $M=(y_1,\ldots,y_n)$ is a point in $\ZR^n$. Let $V$ be a subset of ${\rm Poly}_d(\ZR^n)$ such that $P(M)\not=0$ for any polynomial $P\in V$. Then $V$ is dense in ${\rm Poly}_d(\ZR^n)$. Also, the complement of $V$ with respect to ${\rm Poly}_d(\ZR^n)$, $V^c$, has measure 0.
\end{lemma}
\begin{proof}
Note that as a vector space, ${\rm Poly}_d(\ZR^n)$ has dimensions $\binom{n+d}{n}$. If $P\in {\rm Poly}_d(\ZR^n)$ satisfying $P(M)=0$, then $P(x)$ must have $(x_i-y_i)$ as its factor, for an $i\in\{1,\ldots,n\}$. This shows that $V^c$ is contained in the union of $n$ affine subspaces of ${\rm Poly}_d(\ZR^n)$, each of which has dimensions $\binom{n+d-1}{n}<\dim({\rm Poly}_d(\ZR^n))$. Therefore, $V^c$ has measure 0, implying that $V$ is dense in ${\rm Poly}_d(\ZR^n)$.
\end{proof}

We also need a generalization of B\'ezout's theorem.
\begin{lemma}
Let $P_1,\ldots,P_k\in{\rm Poly}_d(\ZR^n)$ be a sequence of non-singular polynomials, such that $Z(P_1,\ldots,P_k)$ is a transverse complete intersection. If we set $\#(Z(P_1,\ldots,P_k))$ be the number of connected components of $Z(P_1,\ldots,P_k)$, then 
\begin{equation}
    \label{generalized-bezout}
    \#(Z(P_1,\ldots,P_k))\leq\prod_{i=1}^k\deg(P_k).
\end{equation}
\end{lemma}
\begin{proof}
We embed ${\rm Poly}_d(\ZR^n)$ into ${\rm Poly}_d(\ZC^n)$. Let $Z_1,\ldots,Z_l$ be the irreducible components of $Z(P_1,\ldots,P_k)$, under the algebraic closed field $\ZC$. Then, by repeatedly using Theorem 7.7 in \cite{Hartshorne}, we have
\begin{equation}
\label{big-Bezout}
    l\leq \prod_{i=1}^k\deg(P_k).
\end{equation}
Since a irreducible component of $Z(P_1,\ldots,P_k)$ under $\ZC$ is always a irreducible component of $Z(P_1,\ldots,P_k)$ under $\ZR$, and since under $\ZR$, a irreducible component of $Z(P_1,\ldots,P_k)$ is always a connected component of $Z(P_1,\ldots,P_k)$, we can deduce \eqref{generalized-bezout} from \eqref{big-Bezout}.
\end{proof}

Now we can begin our proof of Proposition \ref{final-partition}. The idea is to use Theorem \ref{polynomial-sandwich}, Lemma \ref{dense-lemma} and Lemma \ref{dense-lemma-2} inductively to construct a desired polynomial. In the first step, using Theorem \ref{polynomial-sandwich}, we can find a polynomial $Q_1$ such that
\begin{equation}
    \int_{\{Q_1>0\}}g=\int_{\{Q_1<0\}}g=2^{-1}\int g.
\end{equation}
If we take $M=\ZR^n$ in Lemma \ref{dense-lemma}, we know that non-singular polynomials are dense in ${\rm Poly}_d(\ZR^n)$. Thus, by dominate convergence theorem, we can find a non-singular polynomial $P_1$ such that
\begin{equation}
    (1-d^{-10n})\int_{\{P_1>0\}}g\leq\int_{\{P_1<0\}}g\leq(1+d^{-10n})\int_{\{P_1>0\}}g.
\end{equation}

Next, we let $g_1=\Id_{\{P_1>0\}}g$ and $g_2=\Id_{\{P_1<0\}}g$, and apply Theorem \ref{polynomial-sandwich} again so that we can find a polynomial $Q_2$ satisfying 
\begin{equation}
    \int_{\{Q_2>0\}}g_i=\int_{\{Q_2<0\}}g_i=2^{-1}\int g_i
\end{equation}
for $i=1,2$. We take $M=Z(P_1)$ in Lemma \ref{dense-lemma}, combining the fact that non-singular polynomials are dense in ${\rm Poly}_d(\ZR^n)$ and the dominate convergence theorem to conclude the following: There is a non-singular polynomial $P_2$ such that $Z(P_1,P_2)$ is a transverse complete intersection, and for $i=1,2$,
\begin{equation}
    (1-d^{-10n})\int_{\{P_2>0\}}g_i\leq\int_{\{P_2<0\}}g_i\leq(1+d^{-10n})\int_{\{P_2>0\}}g_i.
\end{equation}

Generally, suppose that we are in the inductive step $k$. Since a countable union of sets of measure zero has measure zero, we can similarly use Theorem \ref{polynomial-sandwich}, Lemma \ref{dense-lemma} and Lemma \ref{dense-lemma-2} to find a non-singular polynomial $P_k$ so that
\begin{enumerate}
    \item Up to a factor $(1-d^{-10})$, $Z(P_k)$ bipartite the $L^1$ norm of $g$ in each of the precedent $2^{k-1}$ cells \begin{equation}
    \nonumber
        \bigcap_{j=1}^{k-1}\{(-1)^\si P_j>0\},\hspace{1cm}\si\in\{0,1\}.
    \end{equation}
    \item $Z(P_{i_1},\ldots,P_{i_l})$ is a transverse complete intersection for any sequence of polynomials $P_{i_1},\ldots,P_{i_l}\in \{P_1,\ldots,P_k\}$.
\end{enumerate}
We point out that Lemma \ref{dense-lemma-2} is used in (2) when $l\geq n$.

Theorem \ref{polynomial-sandwich} tells us the degree of $P_k$ is $O(2^{k/n})$, which implies the degrees of the product $P_1\cdots P_k$ is $O(2^{k/n})$. We choose a natural number $k$ obeying $d/2<2^{k/n}\leq d$, and let $P_0=P_1\cdots P_k$. As a result, $\deg(P_0)\sim d$, $Z(P_0)$ partition $\ZR^n$ into $\sim d^n$ many cells $U$, and up to a factor 2, the $L^1$ norm of $g$ in each cell $U$ is the same. We denote by $\cu$ the collection of these cells $U$. Since the $L^1$ function $g$ is supported in $B^n(0,R)$, any cell $U\in\cu$ must intersects the ball $B^n(0,R)$.

Finally, we want to use some appropriate hyperplanes to cut the ball $B^n(0,2R)$ into $\sim d^n$ many cells of diameter $\sim R/d$. Specifically, for each variable $x_i$, $i=1,\ldots,n$, we choose $4d$ polynomials $Q_{ij}(x)=x_i-c_{ij}$, $j=1,\ldots,4d$, such that $Rd^{-1}/2\leq|c_{ij}-c_{ij'}|\leq Rd^{-1}$ if $|j-j'|=1$, and $Z(P_{l_1},\ldots,P_{l_h})$ is a transverse complete intersection for any $P_{l_1},\ldots,P_{l_h}\in \{P_1,\ldots,P_k, \{Q_{ij}\}_{i,j}\}$. The existence of $Q_{ij}$ is guaranteed by Lemma \ref{dense-lemma}, with $V$ taken to be the set $\{x_i+c:c\in\ZR\}$. Note that $\{Z(Q_{ij})\}_{i,j}$ cuts $B^n(0,2R)$ into $\sim d^n$ many cells with diameter $\leq Rd^{-1}$. The polynomial $P=P_0\cdot(\prod_{i,j}Q_{ij})$ is what we are looking for in Proposition \ref{final-partition}.

\vspace{3mm}

We claim that the number of connected components of $\ZR^n\setminus Z(P)$ is $O(d^n)$. Supposing at first the claim is justified. Let $\co$ be the collection of connected components of $\ZR^n\setminus Z(P)$ contained in $B^n(0,3R/2)$. Then by the construction of the polynomial $P$ above, each cell $O\in\co$ is contained in a cube in $B^n(0,2R)$ of diameter $\sim Rd^{-1}$. Notice that any set $U\in\cu$ contains at least one cell in $\co$. Since $|\cu|\sim d^n$, by pigeonholing, there exists a subcollection $\cu'\subset \cu$ with $|\cu'|\sim d^n$ such that any $U\in \cu'$, $U$ contains $O(1)$ many cells in $\co$. Since $\int_U g\sim d^{-n}\int g$, by pigeonholing again, for each $U\in \cu'$, we can find a cell $O\in\co$, such that $O\subset U$ and $\int_{O}g\sim d^{-n}\int g$. This proves Proposition \ref{final-partition}.

It remains to prove $|\co|=O(d^n)$. The idea is to calculate the number of appropriate $i$-th dimensional boundaries of the sets in $\co$. For convenience, we rewrite the sequence of polynomials $\{P_1,\ldots,P_k, \{Q_{ij}\}_{i,j}\}$ as $\{P_1,\ldots,P_k, P_{k+1},\ldots, P_m\}$, so that $\sum_{l\leq m}\deg(P_l)=\deg(P)\sim d$. For $i=0,1,\ldots,n-1$, we let $\Pi_i$ be the collection of $i$-th dimensional manifolds defined as 
\begin{equation}
    \Pi_i=\{Z(P_{l_1},\ldots,P_{l_{n-i}})\not=\varnothing:P_{l_1},\ldots,P_{l_{n-i}}\in\{P_1,\ldots, P_m\}\}.
\end{equation}
Observe that an element in $\Pi_i$ is a union of $i$-th dimensional connected components, each of which may serve as $i$-th dimensional boundaries for an $O\in\co$.

Next, we partition $\co$ into $n$ disjoint subcollections $\co_i$, $i=0,1,\ldots,n-1$, such that for any $O\in \co_i$, the minimal dimensions of the boundaries of $O$ is $i$. Notice that if $O\in\co_i$ has an $i$-th dimensional boundary which is also connected, then this boundary is a subset of an element in $\Pi_i$. Since $Z(P_{l_1},\ldots,P_{l_{h}})$ is a transverse complete intersection, it follows that for two multi-indices $(l_1,\ldots, l_h)\not=(l_1'\ldots,l_h')$, $Z(P_{l_1},\ldots,P_{l_{h}})\not=Z(P_{l_1'},\ldots,P_{l_{h}'})$. Thus, for each connected component of an $i$-th dimensional manifold $Z(P_{l_1},\ldots,P_{l_{n-i}})$, there are at most $2^{n-i}$ many cells $O\in\co_i$ that use this connected component as their boundary. This is because there is a one-to-one map between these cells $O$, and the sets\footnote{See also Lemma 1 in \cite{Warren}.}
\begin{equation}
    \bigcap_{j=1}^{n-i}\{(-1)^\si P_{l_j}>0\}, \hspace{1cm}\si\in\{0,1\}.
\end{equation}

As a result, recalling $\#(Z(P_{l_1},\ldots,P_{l_{n-i}}))$ is the number of connected components of $Z(P_{l_1},\ldots,P_{l_{n-i}})$, the cardinality of $\co$ is bounded above by
\begin{equation}
    \sum_{i=0}^{n-1}2^{n-i}\sum_{Z(P_{l_1},\ldots,P_{l_{n-i}})\in\Pi_i}\#(Z(P_{l_1},\ldots,P_{l_{n-i}})).
\end{equation}
By \eqref{generalized-bezout}, $\#(Z(P_{l_1},\ldots,P_{l_{n-i}}))\leq\deg(P_{l_1})\cdots\deg(P_{l_{n-i}})$. It follows that
\begin{equation}
    |\co|\leq \sum_{i=0}^{n-1}2^{n-i}\sum_{Z(P_{l_1},\ldots,P_{l_{n-i}})\in\Pi_i}\deg(P_{l_1})\cdots\deg(P_{l_{n-i}}),
\end{equation}
which is bounded by
\begin{equation}
    n2^n\max_i\sum_{P_{l_1},\ldots,P_{l_{n-i}}\in\{P_1,\ldots,P_m\}}\deg(P_{l_1})\cdots\deg(P_{l_{n-i}}).
\end{equation}
We further bound the above quantity by
\begin{equation}
    n2^n\Big(\sum_{l=1}^m\deg(P_l)\Big)^n\lesssim n2^nd^n.
\end{equation}
This shows that $|\co|=O(d^n)$. \qed

\section{Iterated polynomial partitioning algorithm}
We will prove Theorem \ref{broad-theorem} in the next three sections. By homogeneity, we assume $\|f\|_2=1$. We also assume $R>N^{N^N}$ with $N=E_\e^{E_\e}$ for convenience. Consequently, to prove Theorem \ref{broad-theorem}, it suffices to consider those points $x\in B_R$ such that ${\rm Br}_\al Sf(x)\geq R^{-10}$, and assume $\|{\rm Br}_\al Sf\|_{L^p(B_R)}\geq R^{-10}$.
\vspace{3mm}

\subsection{One-step polynomial partitioning}\hfill

We apply Proposition \ref{final-partition} to the $L^1$ function $|{\rm Br}_\al Sf|^p$, so that we have a polynomial $P$ of degree $O(d)$ and a collection of cells $\co$ with $|\co|\sim d^3$, satisfying the conditions that
\begin{equation}
    \int_{O}|{\rm Br}_\al Sf(x)|^p\sim d^{-3}\int_{B_R}|{\rm Br}_\al Sf(x)|^p
\end{equation}
for any cell $O\in\co$, and $O$ is contained in an $Rd^{-1}$ cube in $B_{2R}$. Let $W$ be the $R^{1/2+\be}$ neighborhood of $Z(P)$ contained in $B_{2R}$, and let $O'=O\setminus W$, so
\begin{equation}
    \int_{B_R}|{\rm Br}_\al Sf|^p\lesssim\int_W|{\rm Br}_\al Sf|^p+\sum_{O\in\co}\int_{O'}|{\rm Br}_\al Sf|^p.
\end{equation}

We do wave packet decomposition for the function $Sf$, and use $f^O$ to denote the sum of wave packets $f_{T}$ with $O'\cap R^\be T\not=\varnothing$. That is, if we let $\ZT_{O}$ be the collection of tubes such that $O'\cap R^\be T\not=\varnothing$, then $f^O=\sum_{T\in\ZT_{O}} f_T$. As proved in \cite{Guth-restriction-R3}, we have
\begin{lemma}
\label{1-2-broadness-cell}
For any point $x\in O'$, we have
\begin{equation}
    {\rm Br}_\al Sf(x)\leq {\rm Br}_{2\al} f^{O}(x).
\end{equation}
\end{lemma}

\vspace{3mm}

Following the idea in \cite{Guth-restriction-R3}, we let $\ZB=\{B_k\}$ be a collection of finitely overlapping $R^{1-\de}$ cubes in $B_{2R}$. For each $B_k$, we define two subcollections of tubes, $\ZT_{k,tang}$ and $\ZT_{k,trans}$ by
\begin{definition}
$\ZT_{k,tang}$ is the set of tubes $T\in\ZT$ obeying the following two conditions:
\begin{enumerate}
    \item [$\bullet$] $R^\be T\cap W\cap B_k\not=\varnothing$.
    \item [$\bullet$] If $z$ is any non-singular point of $Z(P)$ lying in $10 B_k\cap 10R^\be T$, then 
    \begin{equation}
        |{\rm Angle}(v(T),T_zZ(P)|\leq R^{-1/2+\de}.
    \end{equation}
\end{enumerate}
\end{definition}
\begin{definition}
$\ZT_{k,trans}$ is the set of tubes $T\in\ZT$ obeying the following two conditions:
\begin{enumerate}
    \item [$\bullet$] $R^\be T\cap W\cap B_k\not=\varnothing$.
    \item [$\bullet$] There exists a non-singular point $z$ of $Z(P)$ in $10 B_k\cap 10R^\be T$, so that 
    \begin{equation}
    \label{transverse-condition}
        |{\rm Angle}(v(T),T_zZ(P)|>R^{-1/2+\de}.
    \end{equation}
\end{enumerate}
\end{definition}
We let $f_{k,trans}$ be the sum of wave packets $f_{T}$ that $T\in\ZT_{k,trans}$, and let $f_{k,tang}$ be the sum of wave packets $f_T$ that $T\in\ZT_{k,tang}$. We also let $f_{\tau, k, trans}$ be the sum of wave packets $f_{\theta,T}$ such that $T\in\ZT_{k,trans}$ and $\theta\subset2\tau$. For $J$ being a subset of $\cT$, we define $f_{J, k, trans}$ to be the sum of the functions $f_{\tau, k, trans}$ that $\tau\in J$. We can similarly define $f_{\tau,k,tang}, f_{J,k,tang}$.

\vspace{3mm}

Next, we consider the following bilinear operator
\begin{equation}
    {\rm Bil}( f_{k,tang})(x)=\sum_{\tau_1\not=\tau_2,\tau_i\in\cT}|f_{\tau_1,k,tang}(x)|^{1/2}| f_{\tau_2,k,tang}(x)|^{1/2}.
\end{equation}
The following lemma is essentially proved in \cite{Guth-restriction-R3}.
\begin{lemma}
\label{1-2-broadness-wall}
Assume $\al<10^{-5}$. Let $\cT_\tau$ be the subset of $\cT$ that $\cT_\tau\sqcup\tau=\cT$. Then for any $x\in B_k\cap W$, 
\begin{equation}
\label{trans-tangent-broadness}
    {\rm Br}_\al Sf(x)\leq 2\big(\sum_{\tau\in\cT}{\rm Br}_{2\al} f_{\cT_\tau,k,trans}(x)+K^{2}{\rm Bil}(f_{k,tang})(x)\big).
\end{equation}
\end{lemma}
It follows that if we define $\cT(B_k)\subset\cT$ to be the set satisfying
\begin{equation}
    \|\Id_{B_k\cap W}{\rm Br}_{2\al} f_{\cT(B_k),trans}\|_p=\max_{\tau\in\cT}\|{\Id_{B_k\cap W}\rm Br}_{2\al} f_{\cT_\tau,k,trans}\|_p,
\end{equation}
then we have
\begin{equation}
\nonumber
    \int_{B_k\cap W} {\rm Br}_\al Sf^p\leq CK^2\int_{B_k\cap W}|{\rm Br}_{2\al} f_{\cT(B_k),trans}|^p+CK^2\int_{B_k\cap W} {\rm Bil}(f_{k,tang})^p.
\end{equation}

When considering all $B_k\in\ZB$, the transverse part in the right hand side of the inequality above is not good enough for iteration. This drawback motivates us to use the dyadic pigeonholing trick to find an appropriate subset of $\ZB$, as shown in the next lemma.
\begin{lemma}
\label{choice-of-mu}
There is a dyadic value ${\mu_1}$ and a set $\ZB_{\mu_1}\subset\ZB$, such that for any $B_k\in\ZB_{\mu_1}$,
\begin{equation}
    \int_{B_k\cap W}|{\rm Br}_{2\al} f_{\cT(B_k),trans}|^p\sim \mu_1.
\end{equation}
Also, we have
\begin{equation}
\nonumber
    \sum_{B_k\in\ZB_{{\mu_1}}}\int_{B_k\cap W}|{\rm Br}_{2\al} f_{\cT(B_k),trans}|^p\gtrsim (\log R)^{-1}\sum_{B_k\in\ZB}\int_{B_k\cap W}|{\rm Br}_{2\al} f_{\cT(B_k),trans}|^p.
\end{equation}
\end{lemma}
\begin{proof}
We sort 
\begin{equation}
    \int_{B_k\cap W}|{\rm Br}_{2\al} f_{\cT(B_k),trans}|^p
\end{equation}
according to its magnitude. For a dyadic number $\mu\lesssim1$, we let $\ZB_{\mu}$ be the collection of $B_k$ such that
\begin{equation}
    \int_{B_k\cap W}|{\rm Br}_{2\al} f_{\cT(B_k),trans}|^p\sim \mu.
\end{equation}
Notice that we only need to consider those $\mu\geq R^{-C}$. Otherwise, the contributions from the transverse part is negligible. Consequently, there are $O(\log R)$ many possible $\mu$. By pigeonholing, we can choose a dyadic value $\mu_1$ and the corresponding set $\ZB_{\mu_1}$ such that
\begin{equation}
\nonumber
    \sum_{B_k\in\ZB_{\mu_1}}\int_{B_k\cap W}|{\rm Br}_{2\al} f_{\cT(B_k),trans}|^p\gtrsim (\log R)^{-1}\sum_{B_k\in\ZB}\int_{B_k\cap W}|{\rm Br}_{2\al} f_{\cT(B_k),trans}|^p.
\end{equation}
\end{proof}

We also have the following two lemmas essentially concerning the incidences between tubes and domains. Their proofs use the fundamental theorem of algebra, and the Lemma 4.1 in \cite{Guth-restriction-R3} with appropriate choices of $\rho$ and $a$.  
\begin{lemma}
\label{tranverse-lemma2}
Each tube $T$ belongs to $\ZT_O$ for $O(d)$ many $O\in\co$. Combining the almost $L^2$-orthogonality  \eqref{L2-orthogonality}, we have
\begin{equation}
    \sum_{O\in\co}\|f^O\|_2^2\lesssim d\|Sf\|_2^2.
\end{equation}
\end{lemma}
\begin{lemma}\footnote{See Lemma 3.5 in \cite{Guth-restriction-R3}.}
\label{transverse-lemma1}
Each tube $T$ belongs to $\ZT_{k,trans}$ for $O({\rm Poly}(d))$ many $k$. Combining the almost $L^2$-orthogonality  \eqref{L2-orthogonality}, we have
\begin{equation}
    \sum_{B_k\in\ZB_{\mu_1}}\|f_{\cT(B_k),trans}\|_2^2\lesssim{\rm Poly}(d)\|Sf\|_2^2.
\end{equation}
\end{lemma}

We conclude what we did above into a lemma for clarity:
\begin{lemma}
\label{one-step-partitioning}
There exists:
\begin{enumerate}
    \item A polynomial $P$ with degree $O(d)$;
    \item A collection of cells $\co$ with $|\co|\sim d^3$ such that any $O\in\co$, $O$ is contained in an $Rd^{-1}$ cube in $B_{2R}$;
    \item A wall $W=N_{R^{1/2+\be}}Z(P)\cap B_{2R}$, two collections of $R^{1-\de}$ cubes $\ZB$, $\ZB_{\mu_1}$, and a number $N_1=|\ZB_{\mu_1}|$;
    \item Several functions $f^O$, $f_{k,tang}$, $f_{k,trans}$, $f_{\tau,k,tang}$, $f_{\cT(B_k),trans}$ and several collections of tubes $\ZT_{O}$, $\ZT_{k,trans}$, $\ZT_{k,tang}$, $\ZT_{\theta,k,tang}$, etc.;
\end{enumerate}
such that
\begin{equation}
    \sum_{O\in\co}\|f^O\|_2^2\lesssim d\|Sf\|_2^2;
\end{equation}
\begin{equation}
    \sum_{B_k\in\ZB_{\mu_1}}\|f_{\cT(B_k),trans}\|_2^2\lesssim {\rm Poly}(d)\|Sf\|_2^2.
\end{equation}
What is more, we have 
\begin{eqnarray}
\label{cell-trans-tang1}
\nonumber
    \int_{B_R}\!\!|{\rm Br}_\al Sf|^p\!\!\!\!\!\!\!\!\!\!\!&&\leq C_\e K^{10}\sum_{O\in\co}\int_{\co'}|{\rm Br}_{2\al} f^O|^p+C_\e K^{10}\sum_{B_k\in\ZB}\int_{B_k\cap W}{\rm Bil}(f_{k,tang})^p\\ \nonumber
    &&+~C_\e K^{10}\sum_{B_k\in\ZB_{\mu_1}}\int_{B_k\cap W}|{\rm Br}_{2\al} f_{\cT(B_k), trans}|^p\\ 
    &&:=I+II+III.
\end{eqnarray}
\end{lemma}

\subsection{Two-step polynomial partitioning}\hfill

We start from \eqref{cell-trans-tang1}. We set $\co_1=\co$, $W^1=W$ and $\ZB^1=\ZB$ in \eqref{cell-trans-tang1} so that $\co_1$ represents cells deriving from the one-step polynomial partitioning. In order to adapt the idea in the last subsection smoothly, we will alter \eqref{cell-trans-tang1} a little. For each cell $O_1\in\co_1$, let $\bar B_{O_1}$ be the $Rd^{-1}$ cube containing $O_1$ and let $B_{O_1}=2\bar B_{O_1}$. We point out that for different $O_1\in\co_1$, the associated cube $B_{O_1}$ can highly overlap. 

We choose two smooth cutoff functions $\vp_{O_1}$, $\vp_{B_k}$ associated to each $O_1\in\co_1$ and each $B_k\in\ZB_{\mu_1}$, respectively. Here we require:
\begin{enumerate}
    \item $\vp_{O_1}\geq c$ on $B_{O_1}$ and $\vp_{B_k}\geq c$ on $B_k$.
    \item ${\rm supp}(\wh\vp_{O_1})\subset B(0,3dR^{-1})$, $|\wh{\vp}_{O_1}|\geq c$ on $B(0,dR^{-1})$ and ${\rm supp}(\wh\vp_{B_k})\subset B(0,3R^{\de-1})$, $|\wh{\vp}_{B_k}|\geq c$ on $B(0,R^{\de-1})$ for a small constant $c$. 
\end{enumerate}
Therefore, from the inequality \eqref{cell-trans-tang1} we have
\begin{eqnarray}
\label{cell-trans-tang2}
\nonumber
    \int_{B_R}\!\!|{\rm Br}_\al Sf|^p\!\!\!\!\!\!\!\!\!\!\!&&\leq C_\e K^{10}\sum_{O_1\in\co_1}\int_{O_1'}|{\rm Br}_{2\al} (f^{O_1}\vp_{O_1})|^p+C_\e K^{10}\!\!\!\sum_{B_k\in\ZB^1}\!\int_{B_k\cap W^1}\!\!\!\!{\rm Bil}(f_{k,tang})^p\\ \nonumber
    &&+~C_\e K^{10}\sum_{B_k\in\ZB_{\mu_1}}\int_{B_k\cap W^1}|{\rm Br}_{2\al} (f_{\cT(B_k), trans}\vp_{B_k})|^p\\ 
    &&:=I+II+III.
\end{eqnarray}
We will call $I$ cell term, $II$ tangent term, $III$ transverse term.

\vspace{3mm}

Let us see how to use polynomial partitioning for another time. When we say ``the tangent case dominates", we mean that $II\geq\max\{I,III\}$. We can make similar definitions for the other two cases. 

\vspace{3mm}

${\bf 1.~ Tang:}$ When the tangent case dominates, we will stop and do nothing. 

\vspace{3mm}

${\bf 2.~ Cell:}$ Otherwise, assuming the cell case dominates. We set $r_1=Rd^{-1}$, and let $r=r_1$ for brevity. We also define $\co^1$ by
\begin{equation}
    \co^1=\{O_1':O_1\in\co_1\},
\end{equation}
to record the collection of sets that we encounter at step 1.

Notice that the Fourier support of $f^{O_1}\vp_{O_1}$ is contained in $N_{Cr^{-1}}(\Ga)$. For each $O_1\in\co_1$, we define $g_{O_1}=f^{O_1}\vp_{O_1}$ and let $g=g_{O_1}$ in short. Then, we do wave packet decomposition on $g$ at the scale $r$:
\begin{equation}
    g=\sum_{\om}\sum_{I\in\ZI_\om} g_{\om,I}.
\end{equation}

Next, for each $O_1\in\co_1$, we do polynomial partitioning for $\Id_{O_1'}|{\rm Br}_{2\al}g|^p$ to obtain a polynomial $P$ of degree $O(d)$ and a wall $W=N_{r^{1/2+\be}}Z(P)\cap B_{O_1}$. For some technical reasons, we need to modify the definitions of ``transverse" and ``tangent". Let $\ZB^2=\{B_k\}$ be a collection of finitely overlapping $rR^{-\de}$ cubes that $B_k\subset B_{O_1}$. For each $B_k\in\ZB^2$, we define $\ZI_{k,tang},\ZI_{k,trans}$ by:
\begin{definition}
\label{tangent-definition}
$\ZI_{k,tang}$ is the set of tubes $I\in\ZI$ obeying the following two conditions:
\begin{enumerate}
    \item [$\bullet$] $r^\be I\cap W\cap B_k\not=\varnothing$.
    \item [$\bullet$] If $z$ is any non-singular point of $Z(P)$ lying in $10B_k\cap 10r^\be I$, then 
    \begin{equation}
        |{\rm Angle}(v(I),T_zZ(P)|\leq r^{-1/2}R^\de.
    \end{equation}
\end{enumerate}
\end{definition}
\begin{definition}
$\ZI_{k,trans}$ is the set of tubes $I\in\ZI$ obeying the following two conditions:
\begin{enumerate}
    \item [$\bullet$] $r^\be I\cap W\cap B_k\not=\varnothing$.
    \item [$\bullet$] There exists a non-singular point $z$ of $Z(P)$ in $10 B_k\cap 10r^\be I$, so that
    \begin{equation}
    \label{transverse-condition2}
        |{\rm Angle}(v(I),T_zZ(P)|>r^{-1/2}R^{\de}.
    \end{equation}
\end{enumerate}
\end{definition}
Since $r$ is much larger than $R^\de$, we still can apply the Lemma 4.1 in \cite{Guth-restriction-R3}. Henceforth, similar to Lemma \ref{one-step-partitioning}, we can obtain
\begin{enumerate}
    \item A polynomial $P$ with degree $O(d)$;
    \item A collection of cells $\co_2$ with $|\co_2|\sim d^3$ such that any $O_2\in\co_2$, $O_2$ is contained in an $rd^{-1}$ cube in $B_{O_1}$;
    \item A wall $W^2=N_{r^{1/2+\be}}Z(P)\cap B_O$, two collections of $rR^{-\de}$ cubes $\ZB^2$, $\ZB_{\mu_2}$,  and a number $N_2=|\ZB_{\mu_2}|$;
    \item Several collections of functions $g^{O_2}$, $g_{k,tang}$, $g_{k,trans}$, $g_{\tau,k,tang}$, $g_{\cT(k), trans}$ and several collections of sets $\ZI_{O_2}$, $\ZI_{k,trans}$, $\ZI_{k,tang}$; $\ZI_{\om,k,tang}$ etc. ;
\end{enumerate}
such that
\begin{equation}
    \sum_{O_2\in\co_2}\|g^{O_2}\|_2^2\lesssim d\|g\|_2^2.
\end{equation}
\begin{equation}
    \sum_{B_k\in\ZB_{\mu_2}}\|g_{\cT(B_k),trans}\|_2^2\lesssim {\rm Poly}(d)\|g\|_2^2.
\end{equation}
What is more, we have 
\begin{eqnarray}
\nonumber
\label{one-step-partition-cell}
    \int_{O_1'}\!\!|{\rm Br}_{2\al} g|^p\!\!\!\!\!\!\!\!\!\!\!&&\leq C_\e K^{10}\sum_{O_2\in\co_2}\int_{O_2'}|{\rm Br}_{4\al} g^{O_2}|^p+C_\e K^{10}\sum_{B_k\in\ZB^2}\int_{B_k\cap W^2}{\rm Bil}(g_{k,tang})^p\\ \nonumber
    &&+~C_\e K^{10}\sum_{B_k\in\ZB_{\mu_{2}}}\int_{B_k\cap W^2}|{\rm Br}_{4\al} g_{\cT(B_k),trans}|^p\\  \label{cell-decendent}
    &&:=I_{O_1}+II_{O_1}+III_{O_1}.
\end{eqnarray}

In addition, since $g=f_{O_1}\vp_{O_1}$, and since the Fourier support of $f_{O_1}$ is contained in an $R^{-1}$-neighborhood of $\Ga$, we have from the local $L^2$ estimate \eqref{local-l2} that
\begin{equation}
    \|g\|_2^2\lesssim d^{-1}\|f^{O_1}\|_2^2,
\end{equation}
which implies
\begin{equation}
    \sum_{O_2\in\co_2}\|g^{O_2}\|_2^2\lesssim \|f^{O_1}\|_2^2
\end{equation}
and
\begin{equation}
    \sum_{B_k\in\ZB_{\mu_{2}}}\|g_{\cT(B_k),trans}\|_2^2\lesssim {\rm Poly}(d)d^{-1}\|f^{O_1}\|_2^2.
\end{equation}

The dyadic number $\mu_2$, as well as the number $N_2$, depend on the cell $O_1\in\co_1$. However, since $\mu_2, N_2\in[R^{-C},R]$, we can assume that $\mu_2, N_2$ are two uniform numbers for all $O_1\in\co_1'$, via two dyadic pigeonholing tricks. Here $\co_1'$ is a subset of $\co_1$, such that $|\co_1|\lesssim (\log R)^2|\co_1'|$. Since a loss of $(\log R)^2$ is negligible in our argument, without loss of generality, we assume $\co_1'=\co_1$.

Since the sets $W^2$, $\ZB_{\mu_2}$ depend on $O_1$ implicitly, and since $O_1$ and $O_1'$ are one-to-one, we use $W^2(O_1')$, $\ZB_{\mu_2}(O_1')$ to indicate the dependence between the set $O_1'\in\co^1$ and the two sets $W^2$, $\ZB_{\mu_2}$.

\vspace{3mm}

At this stage, we can call $\sum_{O_1\in\co_1}I_{O_1}$ cell term, $\sum_{O_1\in\co_1}II_{O_1}$ tangent term, and $\sum_{O_1\in\co_1}III_{O_1}$ transverse term. We remark that the number of sets in the cell term is $\sim d^{6}$, and the number of sets in the transverse term is $\sim d^3N_2$.

\vspace{3mm}

${\bf 3. ~Trans:}$ Otherwise, assuming the transverse case dominates. We set $r_1=R^{1-\de}$ and let $r=r_1$ for brevity. Similarly, we define $\co^1$ by
\begin{equation}
    \co^1=\{B\cap W^1:B\in\ZB_{\mu_1}\},
\end{equation}
to record the collection of sets that we encounter at step 1.

For any cube $B\in\ZB_{\mu_1}$, we define $g_B=f_{\cT(B),trans}\vp_{B}$ and let $g=g_B$ in short. Then, we do wave packet decomposition at the scale $r$ for $g$ and polynomial partitioning for $\Id_{W^1\cap B}{\rm Br}_{2\al}g^p$ similarly to have
\begin{enumerate}
    \item A polynomial $P$ with degree $O(d)$;
    \item A collection of cells $\co_2$ with $|\co_2|\sim d^3$ such that any $O_2\in\co_2$, $O_2$ is contained in an $rd^{-1}$ cube in $B_{O_1}$;
    \item A wall $W^2=N_{r^{1/2+\be}}Z(P)\cap B$, two collections of $rR^{-\de}$ cubes $\ZB^2$, $\ZB_{\mu_2}$, and a number $N_2=|\ZB_{\mu_2}|$. Here we assume that the numbers $\mu_2,N_2$ are uniform for $B\in\ZB_{\mu_1}$, for the same reason mentioned at the end of cell case;
    \item Several collections of functions $g^{O_2}$, $g_{k,tang}$, $g_{k,trans}$, $g_{\tau,k,tang}$, $g_{\cT(k), trans}$ and several collections of sets $\ZI_{O_2}$, $\ZI_{k,trans}$, $\ZI_{k,tang}$; $\ZI_{\om,k,tang}$ etc. ;
\end{enumerate}
such that
\begin{equation}
    \sum_{O_2\in\co_2}\|g^{O_2}\|_2^2\lesssim d\|g\|_2^2.
\end{equation}
\begin{equation}
    \sum_{B_k\in\ZB_{\mu_2}}\|g_{\cT(B_k),trans}\|_2^2\lesssim {\rm Poly}(d)\|g\|_2^2.
\end{equation}
What is more, we have 
\begin{eqnarray}
\nonumber
    \int_{B\cap W^1}\!\!|{\rm Br}_{2\al} g|^p\!\!\!\!\!\!\!\!\!\!\!&&\leq C_\e K^{10}\sum_{O_2\in\co_2}\int_{O_2'}|{\rm Br}_{4\al} g^{O_2}|^p+C_\e K^{10}\!\!\!\sum_{B_k\in\ZB^2}\int_{B_k\cap W^2\cap W^1}\!\!\!\!\!\!\!\!{\rm Bil}(g_{k,tang})^p\\ \nonumber
    &&+~C_\e K^{10}\sum_{B_k\in\ZB_{\mu_{2}}}\int_{B_k\cap W^2\cap W^1}|{\rm Br}_{4\al} g_{\cT(B_k),trans}|^p\\  \label{transverse-decendent}
    &&:=I_B+II_B+III_B.
\end{eqnarray}
In addition, we have 
\begin{equation}
    \sum_{O_2\in\co_2}\|g^{O_2}\|_2^2\lesssim dR^{-\de}\|f_{\cT(B),trans}\|_2^2
\end{equation}
and
\begin{equation}
    \sum_{B_k\in\ZB_{\mu_2}}\|g_{\cT(B_k),trans}\|_2^2\lesssim {\rm Poly}(d)R^{-\de}\|f_{\cT(B),trans}\|_2^2
\end{equation}
from the local $L^2$ estimate
\begin{equation}
    \|g\|_2^2\lesssim R^{-\de}\|f_{\cT(B),trans}\|_2^2.
\end{equation}

Similarly, since the sets $W^2$, $\ZB_{\mu_2}$ depend on $B$ implicitly, and since the set $B$ and the set $B\cap W^1$ are one-to-one, we use $W^2(B\cap W^1)$, $\ZB_{\mu_2}(B\cap W^1)$ to indicate the dependence between the set $B\cap W^1\in\co^1$ and the two sets $W^2$, $\ZB_{\mu_2}$.

\vspace{3mm}

Now we can call $\sum_{B\in\ZB_{\mu_1}}I_{B}$ cell term, $\sum_{B\in\ZB_{\mu_1}}II_{B}$ tangent term, $\sum_{B\in\ZB_{\mu_1}}III_{B}$ transverse term. We remark that the number of sets in the cell term is $\sim d^{3}N_1$, and the number of sets in the transverse term is $\sim N_1N_2$. 

\subsection{Iterated polynomial partitioning}\hfill

Generally, assume $m\geq2$ and assume that we have done polynomial partitioning for $m$ times. Then we do the following:
\begin{enumerate}
    \item If the cell case dominates, and if $r_{m-1}d^{-1}>R^{\e/4}$, we first set $r_m=r_{m-1}d^{-1}$. Then, we define $\co^m$ by
    \begin{equation}
        \co^m=\{O_m':O_m\in\co_m\},
    \end{equation}
    to record the sets we encounter at the step $m$.
    After that, similar to what we did in the last subsection, we do wave packet decomposition for each cell function $g_{O_m}$ at the scale $r_{m}$, and do polynomial partitioning on each function $\Id_{O_m'}|{\rm Br}_{2^m}g_{O_m}|^p$. As a result, we will have collections of functions, sets, tubes, as well as several estimates for the step $m$.
    \item If the transverse case dominates, and if  $r_{m-1}R^{-\de}>R^{\e/4}$, we similarly set $r_m=r_{m-1}R^{-\de}$ and define 
    \begin{equation}
        \co^m:=\{B\cap O\cap W^{m}(O):B\in\ZB_{\mu_{m}}(O),~O\in\co^{m-1}~{\rm is~the~parent~of~}B\},
    \end{equation}
    to record the sets we encounter at the step $m$.
    Then, we do wave packet decomposition for each transverse function $g_B$ at the scale $r_{m}$, and do polynomial partitioning on each function $\Id_{B\cap O\cap W^m(O)}|{\rm Br}_{2^m}g_B|^p$. As a result, we will have collections of functions, sets, tubes, as well as several estimates for the step $m$.
    \item If we are in neither of the two cases above, we stop.
\end{enumerate}

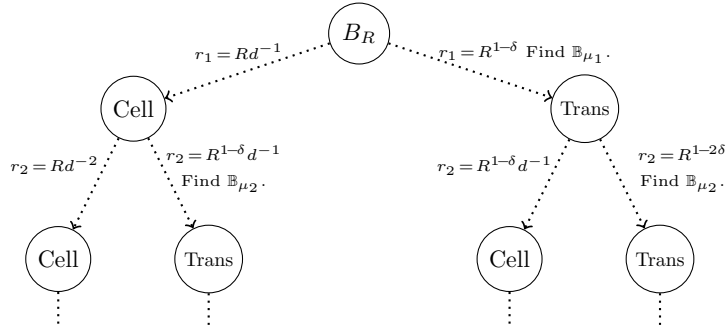
\begin{figure}
\begin{tikzpicture}

\node[shape=circle,draw=black,scale=.9] (0) at (0,0) {$B_R$};
\node[shape=circle,draw=black,scale=.9] (1-1) at (-3,-1) {Cell};
\node[shape=circle,draw=black,scale=.75] (1-2) at (3,-1) { Trans};
\node[shape=circle,draw=black,scale=.9] (2-1) at (-4,-3) {Cell};
\node[shape=circle,draw=black,scale=.75] (2-2) at (-2,-3) { Trans};
\node[shape=circle,draw=black,scale=.9] (2-3) at (2,-3) {Cell};
\node[shape=circle,draw=black,scale=.75] (2-4) at (4,-3) { Trans};

\node (3-1) at (-4,-4) {};
\node (3-2) at (-2,-4) {};
\node (3-3) at (2,-4) {};
\node (3-4) at (4,-4) {};

\path [->,draw, dotted, thick](0) edge node[left] {} (1-1);
\path [->,draw, dotted, thick](0) edge node[left] {} (1-2);
\path [->,draw, dotted, thick](1-1) edge node[left] {} (2-1);
\path [->,draw, dotted, thick](1-1) edge node[left] {} (2-2);
\path [->,draw, dotted, thick](1-2) edge node[left] {} (2-3);
\path [->,draw, dotted, thick](1-2) edge node[left] {} (2-4);
\path [draw, dotted, thick](2-1) edge node[left] {} (3-1);
\path [draw, dotted, thick](2-2) edge node[left] {} (3-2);
\path [draw, dotted, thick](2-3) edge node[left] {} (3-3);
\path [draw, dotted, thick](2-4) edge node[left] {} (3-4);

\node at (-1.6,-0.25) {\tiny $r_1\!=\!Rd^{-1}$};
\node at (2.2,-0.25) {\tiny $r_1\!=\!R^{1\!-\!\de}$ Find $\ZB_{\mu_1}$.};
\node at (-4.05,-1.75) {\tiny $r_2\!=\!Rd^{-2}$};
\node at (-1.8,-1.6) {\tiny $r_2\!=\!R^{1\!-\!\de}d^{-1}$};
\node at (-1.8,-2) {\tiny Find $\ZB_{\mu_2}$.};
\node at (1.8,-1.75) {\tiny $r_2\!=\!R^{1\!-\!\de}d^{-1}$};
\node at (4.3,-1.6) {\tiny $r_2\!=\!R^{1\!-\!2\de}$};
\node at (4.3,-2) {\tiny Find $\ZB_{\mu_2}$.};

\end{tikzpicture}
\caption{A diagram for the iterated polynomial partitioning algorithm.}
\label{figure-1}
\end{figure}

Supposing when we stop, we find that we have run polynomial partitioning for cell case for $s_c$ times, and transverse case for $s_t$ times. We let $s=s_c+s_t$ and let $r=Rd^{-s_c}R^{-s_t\de}$ for brevity. 

Notice that when we stop, we will have a collection of sets $\cf=\{F_j\}$. We explain what $\cf$ is here: Suppose that we stop because $rd^{-1}<R^{\e/4}$, then $\cf$ is the set $\co^s$. Otherwise, suppose that we stop at the tangent case, then $\cf$ is the set defined as
\begin{equation}
    \cf:=\{B\cap O\cap W^{s+1}(O):B\in\ZB^{s+1}(O),~O\in\co^s~{\rm is~the~parent~of~}B\}.
\end{equation}
Here $W^{s+1}$ and $\ZB^{s+1}$ are derived from the last polynomial partitioning. We remark that we only encounter the tangent case at the last step.

The diameter of $F_j$ has an upper bound $Rd^{-s_c}R^{-s_t\de}$ when we stop because $r\leq R^{\e/4}$, and $Rd^{-s_c}R^{-(s_t+1)\de}$ when we stop at the tangent case. Also, if we assume that we are facing the transverse case at the step $m_l$, we have an approximation of the cardinality of $\cf$ with
\begin{equation}
\label{captical-f}
    |\cf|\sim  R^{O(\de)}d^{3n}\prod_{l=1}^{s_t}|N_{m_l}|.
\end{equation}

Since $d=R^{\e^6}$, $\de=\e^2$, we get $s\leq\e^{-6}$. As a result, $2^s\al=2^s K^{-\e}\leq\e^{\e^{-100}}<10^{-10}$, which justifies the validity of \eqref{trans-tangent-broadness} throughout our iterated polynomial partitioning algorithm.

\vspace{3mm}
At the end of this section, we will obtain \eqref{broad-estimate} for stopping because $r\leq R^{\e/4+\de}$. We will discuss the tangent case in the next section.

For each $F_j$, from the iteration algorithm we know that there is a function $g_j$, whose Fourier support is contained in the ball $B^3(0,2)$, such that
\begin{equation}
\nonumber
    \int_{B_R}|{\rm Br}_\al Sf|^p \leq (C_\e K^{10})^{s}\sum_{F_j\in\cf}\int_{F_j}{\rm Br}_{2^s\al}(g_{j})^p\sim(C_\e K^{10})^{s} |\cf|\int_{F_j}{\rm Br}_{2^s\al}(g_{j})^p.
\end{equation}
We also have an $L^2$ estimate
\begin{equation}
    \sum_{F_j\in\cf}\|g_{j}\|_2^2\lesssim \big({\rm Poly}(d)\big)^{s_t} R^{-s_t\de}\|Sf\|_2^2.
\end{equation}
Since $\|g_j\|_\infty\lesssim1$, we obtain $\|{\rm Br}_{2^s\al}g_j\Id_{F_j}\|_p^p\lesssim\|g_j\|_2^p$. Hence by pigeonholing, we can find a particular set $F_j$ such that
\begin{eqnarray}
    \int_{B_R}|{\rm Br}_\al Sf|^p &\lesssim& (C_\e K^{10})^{s}|\cf|\int_{F_j}{\rm Br}_{2^s\al}(g_{j})^p\\ \nonumber
    &\lesssim&(C_\e K^{10})^{s}\big({\rm Poly}(d)\big)^{\frac{ps_t}{2}} R^{-\frac{ps_t\de}{2}}|\cf|^{1-\frac{p}{2}}\|Sf\|_2^p.
\end{eqnarray}
The coefficient $(C_\e K^{20})^{s}\big({\rm Poly}(d)\big)^{\frac{ps_t}{2}}$ is bounded above by $O(R^{\e^2})$, so
\begin{equation}
    \int_{B_R}|{\rm Br}_\al Sf|^p\leq C_\e R^{\e^2} R^{-\frac{ps_t\de}{2}}|\cf|^{1-\frac{p}{2}}\|Sf\|_2^p.
\end{equation}
Since $|\cf|$ has a lower bound in \eqref{captical-f} that $|\cf|\gtrsim d^{3s_c}$, and since $d^{s_c}R^{s_t\de}=Rr^{-1}\geq R^{1-{\e/4}}$, combining with $\|Sf\|_2\leq\|f\|_2\leq R^{\frac{3}{2}-\frac{3}{p}}\|f\|_p$ which is obtained by Plancherel and H\"older's inequality, we have for $3<p<10/3$,
\begin{equation}
    \int_{B_R}|{\rm Br}_\al Sf|^p\leq C_\e R^{\e^2+p\e/4}R^{p-3}\|f\|_p^p.
\end{equation}
Taking the $p$-th root to both sides we get \eqref{broad-estimate}. \qed

\section{Tangential contributions}
Suppose that the iterated polynomial partitioning algorithm stops at the tangent case. Then for each $F_j\in\cf$, we have a polynomial $P_j$ of degree $O(d)$, a function $g_{j,tang}$ and a collection of tubes $\ZI_{j,tang}$ associated with it. The polynomial $P_j$, the function $g_j$ and the set $\ZI_{j,tang}$ have the following simple properties: Each tube in $\ZI_{j,tang}$ has dimensions $\sim r^{1/2}\times r^{1/2}\times r$; If we let $B_{F_j}$ be the $rR^{-\de}$ cube containing $F_j$, then the wall $W_j:=N_{r^{1/2+\be}}Z(P_j)\cap B_{F_j}$ contains the set $F_j$; The function $g_{j,tang}$ is a sum of wave packets $g_I$ at the scale $r$. If we define a set of ``directions"
\begin{equation}
\label{la-s-j}
    \La_{s,j} =\{\om_s:\exists I\in\ZI_{j,tang},I~{\rm dual~to~}\om_s\},
\end{equation}
we can write down the explicit formula of each wave packet $g_I$ for $I\in\ZI_{j,tang}$ by $g_I=g_{\om_s,O_s}\Id_{I}^\ast$, $\om_s\in\La_{s,j}$, where
\begin{equation}
\nonumber
    g_{\om_s,O_s}=\Big[\sum_{\substack{\om_{s-1}\\\sim\om_s}}\sum_{\substack{I_{s-1}\\ \in\ZI_{\om_{s-1}}}}\!\!\!\big[\cdots\sum_{\om_1\sim\om_2}\sum_{I_1\in\ZI_{\om_1}}[\sum_{\theta\sim\om_1}f_\theta\sum_{T\in\ZT_\theta'}\Id_T^\ast]\vp_{O_1}\Id_{I_1}^\ast\cdots\big]\Id_{I_{s-1}}^\ast\vp_{O_{s-1}}\Big]\vp_{O_s}.
\end{equation}
Here the sets $O_u\in\co^u$, $1\leq u\leq s$, are uniquely determined by $g_{I}$. The set $\ZI_{\om_u}$ is a collection of tubes dual to $\om_u$ that was chosen from the step $u$ in our iteration. We remark that $\ZI_{\om_u}$ depends on the set $O_u$ implicitly. $\ZT_\theta'\subset\ZT_\theta$ is the collection of tubes we chose when applying polynomial partitioning at the very first time.

We can also write $g_{\om_s,O_s}$ in an inductive manner. For each $1\leq u\leq s-1$, if $O_{u+1}$ is a child of a set $O_u\in\co^{u}$, we let
\begin{equation}
\label{tangent-inductive-1}
    g_{\om_{u+1},O_{u+1}}=\vp_{O_{u+1}}\sum_{\substack{\om_{u} \sim\om_{u+1}}}g_{\om_u,O_u}\sum_{{I_{u}\in\ZI_{\om_{u}}}}\Id_{I_{u}}^\ast;
\end{equation}
For $g_{\om_1}$, we let
\begin{equation}
\label{tangent-inductive-2}
    g_{\om_1,O_1}=\vp_{O_1}\sum_{\theta\sim\om_1}f_\theta\sum_{T\in\ZT_\theta'}\Id_T^\ast.
\end{equation}

Recall that from the iterated polynomial partitioning algorithm, we have the following estimates:
\begin{enumerate}
    \item A broad estimate
    \begin{equation}
    \label{broad-tangent-square-fcn}
        \int_{B_R}|{\rm Br}_\al Sf|^p \leq (C_\e K^{10})^{s}\sum_{F_j\in\cf}\int_{F_j}{\rm Bil}(g_{j,tang})^p.
    \end{equation}
    \item An $L^2$ estimate
    \begin{equation}
    \label{tangent-l2-estimate}
        \sum_{F_j\in\cf}\|g_{j,tang}\|_2^2\lesssim \big({\rm Poly}(d)\big)^{s_t} R^{-s_t\de}\|Sf\|_2^2.
    \end{equation}
\end{enumerate}
Since we only encounter the tangent case at the last step, by a dyadic pigeonholing argument, there is a set $\cf'\subset\cf$ satisfying $|\cf|\lesssim(\log R)^2|\cf'|$, such that when $F_j\in\cf'$, $\|\Id_{F_j}{\rm Bil}(g_{j,tang})\|_p^p$ are the same up to a constant factor, and their sum dominates $(\log R)^{-1}(C_\e K^{10})^{-s}\|\Id_{B_R}{\rm Br}_\al Sf\|_p^p$. Since a loss of $(\log R)^2$ is negligible, without loss of generality, we assume $\cf'=\cf$, so that for all $F_j\in\cf$, $\|\Id_{F_j\cap W}{\rm Bil}(g_{j,tang})\|_p$ are the same up to a constant factor. Also, since each $O_s\in\co^s$ has at most $R^{3\de}$ children $F_j\in\cf$, at a cost of $R^{3\de}$, we can assume that each $O_s$ has exactly one child $F_j$, by a pigeonholing argument.

\vspace{3mm}

The next lemma states that tubes in $\ZI_{j,tang}$ are well localized. Recall that each set $F_j\in\cf$ is contained in the $rR^{-\de}$ cube $B_{F_j}$.
\begin{lemma}
\label{location-lemma}
Fix $F_j\in\cf$. For any $I\in\ZI_{j,tang}$, we define $I^{\rm cut}$ to be the portion of $I$ inside $B_{F_j}$, so $I^{\rm cut}$ is a tube with dimensions $\sim r^{1/2}\times r^{1/2}\times rR^{-\de}$. Then the tube $I^{\rm cut}$ is contained in the set $N_{r^{1/2}R^{O(\de)}}Z(P_j)$.
\end{lemma}
\begin{proof}
For any $I\in\ZI_{j,tang}$, we know from Definition \ref{tangent-definition} that $r^\be I\cap W_j\not=\varnothing$.  We assume that the core line of $I$ is the $e_1$ axis, so there is a point $x\in Z(P_j)$ with $|x'|\leq 2r^{1/2+\be}$. This implies ${\rm Angle}(T_x Z(P),e_1)\leq r^{-1/2}R^\de$. By Taylor's theorem, there is another point $y\in Z(P)$ such that $y_1> x_1$, $|y'|\leq 2r^{-1/2}R^{\de}(y_1-x_1)+|x'|$. Notice that as long as $|y'|\leq10r^{1/2+\be}$, we can replace $x$ by $y$ in the above argument to find another point $z\in Z(P_j)$ with $z_1> y_1$ and $|z'|\leq2r^{-1/2}R^{\de}(z_1-y_1)+|y'|$. Therefore, we can repeat the argument above and eventually find a point $u\in Z(P_j)$ with $u_1-x_1\geq rR^{-\de}$ and $|u'|\leq10r^{1/2+\be}$.

Similarly, on the other direction, we can find a point $v\in Z(P_j)$ with $x_1-v_1\geq rR^{-\de}$ and $|v'|\leq10r^{1/2+\be}$. The existence of the points $u$ and $v$ proves that $I^{\rm cut}\subset N_{r^{1/2}R^{O(\de)}}Z(P_j)$, since $R\geq r$ and $\de>\be$.
\end{proof}

We also have a sharp estimate on each tangent part. Intuitively, we may assume that all tubes in the set $\ZI_{j,tang}$ come from one planar slab of thickness $r^{1/2}$, so that the C\'ordoba-Fefferman $L^4$ observation\footnote{See for example \cite{Fefferman-spherical} and \cite{Cordoba-B-R}.} is applicable. Guth essentially proved in \cite{Guth-restriction-R3} that
\begin{equation}
\label{cordoba-type-lemma}
    \int_{F_j}{\rm Bil}(g_{j,tang})^p\lesssim r^{\frac{5}{2}-\frac{5p}{4}}R^{O(\de)}\|g_{j,tang}\|_2^p.
\end{equation}

We make a remark on a possible refinement when summing up \eqref{cordoba-type-lemma} with all $F_j\in\cf$. We will not use the refinement in the rest of our argument.
\begin{remark}
{\rm
If we sum up all the $F_j\in\cf$ using \eqref{cordoba-type-lemma} directly, then generally we will face a huge loss, because the sets of tubes $\ZI_{j,tang}$ can highly overlap. Instead, we can prove the following refinement
\begin{equation}
\label{refinement-cordoba-type}
    \sum_{F_j\in\cf}\int_{F_j}{\rm Bil}(g_{j,tang})^p\lesssim r^{\frac{5}{2}-\frac{5p}{4}}R^{O(\de)}\Big(\sum_{I\in\ZI_{tang}}\|g_I\|_2^2\Big)^{\frac{p}{2}}.
\end{equation}
Here $\ZI_{tang}$ is the union of $\ZI_{j,tang}$. 

{\emph{ Sketch of Proof.} }We will first prove that
\begin{equation}
\label{cordoba-l4}
    \sum_{F_j\in\cf}\int_{F_j}{\rm Bil}(g_{j,tang})^4\lesssim r^{-\frac{5}{2}}R^{O(\de)}\Big(\sum_{I\in\ZI_{tang}}\|g_I\|_2^2\Big)^2.
\end{equation}

Let $\cq$ be a collection of finitely overlapping $r^{1/2}$ cubes in $B_{2R}$. For each $Q\in\cq$, we define $\ZI_{j,Q,tang}$ to be the collection of tubes in $\ZI_{j,tang}$ that intersect $Q$. For any two different caps $\tau_1,\tau_2\in\cT$, similar to Lemma 3.10 in \cite{Guth-restriction-R3}, we have  that for fixed $F_j\in\cf$,
\begin{equation}
\nonumber
    \int_Q|g_{\tau_1,j,tang}|^2|g_{\tau_1,j,tang}|^2\lesssim R^{O(\de)}\sum_{\om_1\subset2\tau_1}\sum_{\om_2\subset2\tau_2}\sum_{I_1,I_2\in\ZI_{j,tang}}\int_Q|g_{\om_1,I_1}|^2|g_{\om_2,I_2}|^2,
\end{equation}
which is roughly
\begin{equation}
    \sim R^{O(\de)}\sum_{\om_1\subset2\tau_1}\sum_{\om_2\subset2\tau_2}\sum_{I_1,I_2\in\ZI_{j,tang}}\|g_{\om_1,I_1}\|_\infty^2\|g_{\om_2,I_2}\|_\infty^2\int_Q\Id_{I_1}\Id_{I_2}.
\end{equation}
We let $\chi(Q,F_j)$ be the incidence function that $\chi(Q,F_j)=1$ if $Q\cap F_j\not=\varnothing$ and $\chi(Q,F_j)=0$ otherwise. It follows that
\begin{equation}
\nonumber
    \sum_{F_j\in\cf}\int_{F_j}{\rm Bil}(g_{j,tang})^4\lesssim\sum_{\tau_1,\tau_2}\sum_{Q\in\cq}\sum_{F_j\in\cf}\chi(Q,F_j)\int_Q|g_{\tau_1,j,tang}|^2|g_{\tau_1,j,tang}|^2.
\end{equation}
Combining the three inequalities above, we have
\begin{eqnarray}
\label{refined-cordoba-sum}
    &&\sum_{F_j\in\cf}\int_{F_j}{\rm Bil}(g_{j,tang})^4 \lesssim\\ \nonumber
   && R^{O(\de)}\sum_{\tau_1,\tau_2}\sum_{Q\in\cq}\sum_{F_j\in\cf}\sum_{\om_i\subset\tau_i}\sum_{I_i\in\ZI_{j,tang}}\chi(Q,F_j)\|g_{\om_1,I_1}\|_\infty^2\|g_{\om_2,I_2}\|_\infty^2\int_Q\Id_{I_1}\Id_{I_2}.
\end{eqnarray}
Since for two different sets $F_j,F_j'\in\cf$, either $F_j,F_j'$ are $r^{1/2}$ separated, or $F_j,F_j'$ are contained in two disjoint cubes of diameter greater than $r$, respectively. Therefore, for fixed $Q$, there are $O(1)$ many $F_j\in\cf$ that $\chi(Q,F_j)\not=0$. We use this observation to simplify \eqref{refined-cordoba-sum}, and obtain \eqref{cordoba-l4}.

Finally, it is easy to see that
\begin{equation}
\label{simple-l2}
    \sum_{F_j\in\cf}\int_{F_j}{\rm Bil}(g_{j,tang})^2\lesssim \sum_{I\in\ZI_{tang}}\|g_I\|_2^2.
\end{equation}
Therefore, combining \eqref{cordoba-l4} and \eqref{simple-l2}, via H\"older's inequality, we get \eqref{refinement-cordoba-type}.    \qed

}
\end{remark}

\vspace{3mm}

So far we have broken down the origin operator $|{\rm Br}_\al Sf|^p$ into pieces, as shown in  \eqref{broad-tangent-square-fcn}. By the sharp bilinear estimate \eqref{cordoba-type-lemma}, we can pile our pieces up via $L^2$ space. However, the $L^2$ estimate \eqref{tangent-l2-estimate} from our iteration algorithm does not make use of the tangential information. So next we will careful study each $g_{j,tang}$, and find a proper way to sum them up.

As mentioned in the introduction, we are going to consider the Nikodym maximal function. Note that the direction of any wave packet $f_T$ at the scale $R$ is transverse to the plane $\{x_3=0\}$. We define $L_a$ to be the plane $\{x_3=a\}$, and we are going to consider the intersection between certain tubes $T$ at the scale $R$ and the plane $L_a$, for any $|a|\leq 2R$. This is a standard idea to study the Nikodym maximal function.

The next lemma captures some information of the intersection between the plane $L_a$ and tubes at the scale $r$, coming from $\ZI_{j,tang}$. Its proof uses Wolff's hairbrush argument, similar to the Lemma 4.9 in \cite{Guth-restriction-R3}.
\begin{lemma}
\label{incidence-lemma}
Fix $F_j\in\cf$ and recall that $F_j$ is contained in the $rR^{-\de}$ cube $B_{F_j}$. For each $I\in\ZI_{j,tang}$, we set $\mathring{I}$ to be the tube that has the same cross section as $I$, but is infinity in length. Let $\mathring{\ZI}_{j,tang}$ be the collection of all these $\mathring{I}$. Let $\{y_l\}_{l=1}^M$ be a collection of the largest $10Rr^{-1/2}$ separated points in the set 
\begin{equation}
    L_a\bigcap\big\{\cup_{\mr{I}\in\mr{\ZI}_{j,tang}}\mathring{I}\big\}.
\end{equation}
Then, we have an upper bound for $M$ with
\begin{equation}
\label{separate-pt-estimate}
    M\lesssim R^{O(\de)}r^{1/2},
\end{equation}
uniformly for all $|a|\leq 2R$.
\end{lemma}

\begin{proof}
For some technical reasons, we need to consider $\ZI_{j,tang}^{\rm cut}$, the collection of corresponding tubes $I^{\rm cut}$ for all $I\in\ZI_{j,tang}$. Since a loss of $R^{2\de}$ is acceptable in \eqref{separate-pt-estimate}, we assume that the points $\{y_l\}_{l=1}^M$ are in fact $10Rr^{-1/2}R^{\de}$ separated,  satisfying a slightly stronger separation condition. For convenience, we will use $\ZI_{j,tang}$ to denote the set $\ZI_{j,tang}^{\rm cut}$ in the rest of the proof. 

\vspace{3mm}

For each point $y_l$, we pick one tube $\mathring{I}_l$ satisfying $y_l\in\mr{I}_l$, and denote the collection of all these tubes $\mr{I}_l$ by $\mr{\ZI}_{j,a}$. Since the map $\circ:\ZI_{j,tang}\to\mr\ZI_{j,tang}$ is one-to-one, the set $\mr{\ZI}_{j,a}$ induces a set $\ZI_{j,a}\subset\ZI_{j,tang}$. We point out that if two tubes $I_1,I_2\in\ZI_{j,a}$ intersect, then they make an angle $\gtrsim r^{-1/2}R^\de$. From Lemma \ref{location-lemma}, we know that $I\subset N_{r^{1/2}R^{O(\de)}}Z(P_j)\cap B_{F_j}$ for any $I\in\ZI_{j,a}$. If we let $\sq=\{Q\}$ be a collection of finitely overlapping $r^{1/2}$ cubes that $Q\cap N_{r^{1/2}R^{O(\de)}}Z(P_j)\cap B_{F_j}\not=\varnothing$, by Wongkew's theorem\footnote{See \cite{Guth-restriction-R3} Theorem 4.7.}, we get $\#\{Q\}\lesssim R^{O(\de)}rd$.

For any $Q\in \sq$, $I_1,I_2\in\ZI_{j,a}$, we define an incidence function $\chi(Q,I_1,I_2)$ by letting $\chi(Q,I_1,I_2)=1$  if $2I_1\cap 2I_2\cap Q\not=\varnothing$ and $\chi(Q,I_1,I_2)=0$ otherwise. Then, if we let $Q(I)$ be the number of tubes $I\in\ZI_{j,a}$ that intersect $Q$, by Cauchy-Schwarz inequality, 
\begin{equation}
    \sum_{Q,I_1,I_2}\chi(Q,I_1,I_2)\geq\sum_Q Q(I)^2\geq \#\{Q\}^{-1}\Big(\sum_Q Q(I)\Big)^2\geq R^{-O(\de)}d^{-1}|\ZI_{j,a}|^2.
\end{equation}

Suppose that we already have 
\begin{equation}
    |\ZI_{j,a}|\lesssim R^{O(\de)}r^{1/2},
\end{equation}
then there is nothing to prove. Otherwise, we get
\begin{equation}
    \sum_{I_1\not= I_2}\chi(Q,I_1,I_2)\gtrsim R^{-O(\de)}d^{-1}|\ZI_{j,a}|^2.
\end{equation}
Notice that when $I_1\not=I_2$, we have ${\rm Angle}(I_1,I_2)\geq cr^{-1/2}R^{\de}$. Thus, there exists a dyadic value $\nu\in[cr^{-1/2}R^{\de},1]$, so that
\begin{equation}
    \sum_{{\rm Angle}(I_1,I_2)\sim\nu}\chi(Q,I_1,I_2)\gtrsim R^{-O(\de)}d^{-1}|\ZI_{j,a}|^2.
\end{equation}
By pigeonholing, there exists a tube $I_1\in\ZI_{j,a}$ such that for fixed $I_1$,
\begin{equation}
    \sum_{{\rm Angle}(I_1,I_2)\sim\nu}\chi(Q,I_1,I_2)\gtrsim R^{-O(\de)} d^{-1}|\ZI_{j,a}|.
\end{equation}
We fix this tube $I_1$ from now on, and denote by $\ZI_\nu$ the collection of tubes in $\ZI_{j,a}$ who make an angle $\sim \nu$ with respect to $I_1$. Let $H$ be the hairbrush define by the union of tubes in $\ZI_\nu$. Since for any tube $I_2\in\ZI_\nu$, $2I_2\cap 2I_1$ intersects $\sim \nu^{-1}$ many $Q$, we have
\begin{equation}
\label{number-of-tube}
    |\ZI_\nu|\gtrsim\nu R^{-O(\de)}d^{-1}|\ZI_{j,a}|.
\end{equation}
 
Next, we claim that
\begin{equation}
\label{hairbrush-volume-1}
    |\ZI_\nu|r^2\lesssim R^{O(\de)}(\log r)|H|.
\end{equation}

Since any two tubes $I,I'\in\ZI_\nu$ make an angle $\gtrsim r^{-1/2}R^\de$ if $I\cap I'\not=\varnothing$, we can decompose the set $\ZI_\nu$ into $\sim r^{1/2}R^{-\de}$ subsets $\ZI_{\nu,l}$, such that tubes in $\ZI_{\nu,l}$ are contained in a planar slab $P_l$ of thickness $r^{1/2}$. These planar slabs $P_l$ have $I_1$ as their common intersection, and their normal vectors are $r^{-1/2}R^\de$ separated. 

For each tube $I\in\ZI_{\nu,l}$, we let $\ti I$ be the portion $I\setminus N_{cr\nu}(I_1)$ and let $\ti\ZI_{\nu,l}$ be the collection of these $\ti I$. Notice that $\ti I$ is still a tube of length $\sim rR^{-\de}$, since ${\rm Angle}(I_1,I)\sim\nu$. We define $H_l$ to be the union of tubes in $\ti\ZI_{\nu,l}$, so that the sets $\{H_l\}_l$ are finitely overlapped. Consequently, it suffices to show that in one single planar slab,
\begin{equation}
    |\ti\ZI_{\nu,l}|r^2\lesssim R^{O(\de)}(\log r)|H_l|.
\end{equation}

To save notations, we let $\ZI=\ti\ZI_{\nu,l}$, $H=H_l$ and $P=P_l$ in the rest of the proof. For each $I\in\ZI$, we let $\bar{I}$ be the dilated tube which is also contained in the planar slab $P$, satisfying that $\bar{I}$ has the same core line, center and width as $I$, while the length of $\bar{I}$ is $(2C\log r)R^{-\de}r$ for a big constant $C$. We define $\bar\ZI$ to be the collection of these dilated tubes $\bar I$. Thus, if we let $\bar{H}$ be the union of tubes in $\bar\ZI$, we have $|\bar{H}|\leq(2C\log r)^2|H|$. 


Recall that each tube $I\in\ZI$, or its stretch $\mr I$, corresponds to a point $y_I$ on the plane $L_a$. It follows that all the points $y_I$ lie in a rectangle of width $O(r^{1/2})$ on the plane $L_a$, since each $I$ is contained in the planar slab $P$, and since the planar slab $P$ is transverse to the plane $L_a$. Recall again that the points $\{y_I\}$ are $10Rr^{-1/2}R^{\de}$ separated, we can conclude for a fixed tube $J\in\bar\ZI$,
\begin{equation}
    \sum_{J'\in\bar\ZI,J\not=J'}|J\cap J'|\leq 10(\log r)r^2R^{-\de}\leq|J|/2.
\end{equation}
Consequently, since $|\ZI|=|\bar\ZI|$,  
\begin{equation}
    2|H'|\geq |\ZI|\cdot|J|=|\ZI|\cdot(C\log r)r^2R^{-\de}.
\end{equation}
Combining the calculations above we get
\begin{equation}
    |\ZI|r^2\lesssim R^{O(\de)}(\log r) |H|,
\end{equation}
which proves \eqref{hairbrush-volume-1}.

Finally, since $H$ is contained in a fat tube of dimensions $\sim\nu r\times\nu r\times r$, and since $H$ is also contained in the set $N_{r^{1/2}R^{O(\de)}}Z(P_j)$, by Wongkew's theorem, 
\begin{equation}
\label{hairbrush-volume-2}
    |H|\lesssim R^{O(\de)}d\nu r^\frac{5}{2}.
\end{equation}
Henceforth, combining \eqref{number-of-tube}, \eqref{hairbrush-volume-1}, \eqref{hairbrush-volume-2} and $d<R^{O(\de)}$, we conclude $|\ZI_{j,a}|\lesssim R^{O(\de)}r^{1/2}$. This implies \eqref{separate-pt-estimate}.
\end{proof}

\section{Pile things back}
In this section, we will use Lemma \ref{incidence-lemma} to create a backward algorithm, so that we can pile our wave packets which are summed in the function $g_{j,tang}$ back efficiently. The geometric observation inside the backward algorithm is that, for all $F_j\in\cf$, tubes in $\cup\ZI_{j,tang}$ are either well-separated, or ``related" to a bigger tube. 

\vspace{3mm}
Recall in the last section that our iteration algorithm stops at the step $s$, and recall that $\co^{u}$ is the collection of sets at the step $u$. Now we are going backwardly. 

\subsection{One-step backward algorithm}\hfill

${\bf 1.~Cell:}$ Suppose that at the step $s-1$, we were in the cell case. Then for any $O_{s-1}\in\co^{s-1}$, there are $O(d^3)$ many $F_j\in\cf$ that are the offspring of $O_{s-1}$, since we already assumed that each $O_s\in\co^s$ has at most one child $F_j\in\cf$. In the rest of this section, we will use the convention $F_j=O_s$ if $F_j$ is the unique child of $O_s$.

We fix one cell $O:=O_{s-1}\in\co^{s-1}$ at first, and let $\cf_O$ be the collection of sets $F_j$ that are the offspring of $O$. By the inductive relation \eqref{tangent-inductive-1}, we have
\begin{equation}
\label{tangent-inductive-3}
    g_{\om_{s},F_j}=\vp_{F_j}\sum_{\om_{s-1}\sim\om_{s}}g_{\om_{s-1},O}\sum_{{I_{s-1}\in\ZI_{\om_{s-1}}}}\Id_{I_{s-1}}^\ast.
\end{equation}
We can sum up $\|g_{j,tang}\|_2^2$ for $F_j\in\cf_O$ to get
\begin{equation}
    \sum_{F_j\in\cf_O}\|g_{j,tang}\|_2^2\lesssim\sum_{F_j\in\cf_O}\sum_{\om_s\in\La_{s,j}}\int |g_{\om_s,F_j}|^2\Big(\sum_{I\in\ZI_{\om_s,j,tang}}\Id_{I}^\ast\Big).
\end{equation}
Since the Fourier support of $\vp_{F_j}$ is contained in the $Cr^{-1}$ neighborhood of the origin, and since the Fourier support of the function $g_{\om_{s-1},O}\sum_{{I_{s-1}\in\ZI_{\om_{s-1}}}}\Id_{I_{s-1}}^\ast$ is contained in $N_{Cr^{-1}d^{-1}}(\Ga)$, by the local $L^2$ estimate \eqref{local-l2}, 
\begin{equation}
    \int |g_{\om_s,F_j}|^2\Big(\sum_{I\in\ZI_{\om_s,j,tang}}\Id_{I}^\ast\Big)\lesssim d^{-1}\sum_{\om_{s-1}\sim\om_s}\int|g_{\om_{s-1},O}|^2\Big(\sum_{\substack{I_{s-1}\in \\ \ZI_{\om_{s-1},j,tang}}}\Id_{I_{s-1}}^\ast\Big).
\end{equation}
Here $\ZI_{\om_{s-1},j,tang}\subset\ZI_{\om_{s-1}}$ is the collection of tubes $I_{s-1}$ dual to $\om_{s-1}$, that each $I_{s-1}$ contains at least one tube $I_s\in\ZI_{\om_s,j,tang}$, for $\om_{s-1}\sim\om_s$. We use $I_{s-1}\sim I_s$ to indicate that $I_s\subset I_{s-1}$ and $\om_{I_{s-1}}\sim\om_{I_s}$. Notice that $I_{s}\subset I_{s-1}$ does not imply $\om_{I_{s-1}}\sim\om_{I_s}$ in general.

We combine the two inequalities above, to obtain 
\begin{equation}
    \sum_{F_j\in\cf_O}\|g_{j,tang}\|_2^2\lesssim d^{-1}\sum_{\om_{s-1}}\int|g_{\om_{s-1},O}|^2\Big(\sum_{F_j\in\cf_O}\sum_{\substack{I_{s-1}\in \\ \ZI_{\om_{s-1},j,tang}}}\Id_{I_{s-1}}^\ast\Big).
\end{equation}

Next, for a dyadic number $v\in[1,Cd]$, we define $\ZI_{v,O}$ as the collection of tubes, each of which belongs to $\sim v$ many sets $\ZI_{\om_{s-1},j,tang}$. The point here is that $v$ has an upper bound $O(d)$. This follows from the fundamental theorem of algebra, and the fact that the polynomial we used in each partitioning step has degree $O(d)$. See Lemma \ref{tranverse-lemma2} for an identical explanation.

Hence, $\{\ZI_{v,O}\}_v$ forms a cover of the set $\cup_j\ZI_{\om_{s-1},j,tang}$. By two dyadic pigeonholing arguments, there exists a dyadic number $v_{s-1}$ and a set $\bar\co^{s-1}\subset\co^{s-1}$, such that 
\begin{equation}
    \sum_{O\in\co^{s-1}}\sum_{F_j\in\cf_O}\|g_{j,tang}\|_2^2\lesssim (\log R)\sum_{O\in\bar\co^{s-1}}\sum_{F_j\in\cf_O}\|g_{j,tang}\|_2^2,
\end{equation}
and uniformly for $O\in\bar\co^{s-1}$, 
\begin{equation}
\label{l2-backward-not-pigeonholing}
    \sum_{\om_{s-1}}\int|g_{\om_{s-1},O}|^2\Big(\sum_{F_j\in\cf_O}\sum_{\substack{I_{s-1}\in \\ \ZI_{\om_{s-1},j,tang}}}\Id_{I_{s-1}}^\ast\Big)\lesssim(\log R) v_{s-1}\sum_{\substack{I_{s-1}\in \\ \ZI_{v_{s-1},O}}}\|g_{I_{s-1}}\|_2^2.
\end{equation}
Here we use $g_{I}$ in short of $g_{\om}\Id_{I}^\ast$, where $\om$ is the dual of $I$. Since again, a loss of $\log R$ is negligible in our backward algorithm, we can assume $\bar\co^{s-1}=\co^{s-1}$ without loss of generality. We also define $\ZI_{v_{s-1},j}$ to be the subset of $\ZI_{j,tang}$, that for any tube $I_s\in\ZI_{v_{s-1},j}$, there are at least one $I_{s-1}\in\ZI_{v_{s-1},O}$ satisfying $I_{s-1}\sim I_s$. Intuitively, the set $\ZI_{v_{s-1},O}$ is a collection of bigger tubes determined by $\{\ZI_{j,tang}\}_j$, sets of smaller tubes; and the set $\ZI_{v_{s-1},j}$ is a collection of smaller tubes re-determined by $\ZI_{v_{s-1},O}$.

\vspace{3mm}

Now recall the definitions in Lemma \ref{incidence-lemma}. We claim that the number of the maximal $Rr^{-1/2}d^{-1}$ separated points in the set 
\begin{equation}
\label{plane-tube-incidence}
    L_{a,O}:=L_a\bigcap\big\{\cup_{\mr{I}\in\mr{\ZI}_{v_{s-1},O}}\mathring{I}\big\}
\end{equation}
is bounded above by $R^{O(\de)}r^{1/2}d^{5}v_{s-1}^{-1}$, uniformly for all $|a|\leq2R$.

Indeed, if we let $\{y_l\}_{l=1}^{M_{O}}$ be a collection of maximal $Rr^{-1/2}$ separated points in the set $L_{a,O}$, then for each point $y_l$, we can pick one tube $I_l\in\ZI_{v_{s-1},O}$ such that $y\in L_{a,O}\cap\mr{I}_l$. Since the bigger tube $I_{l}$ belongs to the set $\ZI_{v_{s-1},O}$, we can pick $\sim v_{s-1}$ many sets $\ZI_{v_{s-1},j}$ such that for each $j$, we can find a smaller tube $I_{l,j}\in\ZI_{v_{s-1},j}$ with $I_{l}\sim I_{l,j}$.

Next, we pick a point $z_{l,j}$ in $L_a\cap \mr{I}_{l,j}$ for each $I_{l,j}$. Then the number of points we picked, $\#\{z_{l,j}\}$, is  $\sim v_{s-1}M_O$. We will prove that $\#\{z_{l,j}\}\lesssim R^{O(\de)}r^{1/2}d^{3}$, and this would imply our claim.

The proof of $\#\{z_{l,j}\}\lesssim R^{O(\de)}r^{1/2}d^{3}$ follows easily from lemma \ref{incidence-lemma}. Observe that for fixed $j$, the points $z_{l,j}$ are morally $Rr^{-1/2}$ separated. This is because for fixed $j,l$, the point $z_{l,j}$ is contained in a $CRr^{-1/2}$ neighborhood of $L_a\cap\mr{I}_l$. See Figure \ref{figure-2} for an explanation. We know by Lemma \ref{incidence-lemma} that when $j$ is fixed, the number of points $z_{l,j}$ has an upper bound $CR^{O(\de)}r^{1/2}$. This gives $\#\{z_{l,j}\}\lesssim R^{O(\de)}r^{1/2}d^{3}$, by summing up all the $F_j\in\cf_O$, and the fact $|\cf_O|\lesssim d^3$. 

\begin{figure}
\begin{tikzpicture}
  
\draw[thick] (0,8) -- (8,8);
\node at (4,8.5) {$L_a$};
\draw[dashed] (2,.5) rectangle (6,7);
\draw[thick] (3.5,1.5) rectangle (4.5,6.5);
\draw[rotate around={15:(4,1.2)}] (4.15,2.2) rectangle (4.45,3.2);
\draw[rotate around={15:(4,3.2)}] (4.4,3.7) rectangle (4.1,4.7);
\draw[rotate around={15:(4,5.2)}] (3.95,5.2) rectangle (4.25,6.2);
\draw [dotted ] (2,7.1) -- (2,8);
\draw [dotted ] (6,7.1) -- (6,8);
\draw [loosely dashed, rotate around={15:(4,6)}] (3.92,6) -- (3.92,8);
\draw [loosely dashed, rotate around={15:(4,4)}] (4.0,4.5) -- (4.0,8.05);
\draw [loosely dashed, rotate around={15:(4,2.5)}] (3.95,3) -- (3.95,8.2);
\node at (3.93,2.7) {\footnotesize $I_1$};
\node at (4.01,4.2) {\footnotesize $I_2$};
\node at (3.98,5.7) {\footnotesize $I_3$};
\node at (3.98,1.8) {$I$};
\node at (3.98,1) {$N_{CRr^{-1/2}}(I)$};
\node at (9,5) {$I_j$ are smaller tubes.};
\node at (9,4) {$I$ is a larger tube, that $I\sim I_j$.};

\end{tikzpicture}
\caption{Possible relations between tubes at different scales.}
\label{figure-2}
\end{figure}
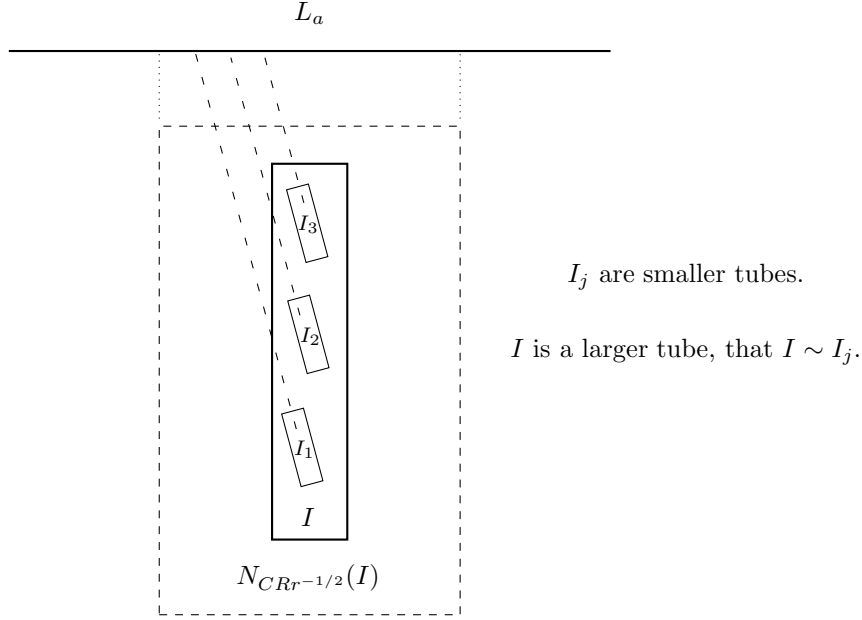

\vspace{3mm}
For each $O\in\co^{s-1}$, we define a sum of wave packets $g_{O,tang}$ as  
\begin{equation}
    g_{O,tang}=\sum_{I_{s-1}\in  \ZI_{v_{s-1},O}}g_{I_{s-1}}.
\end{equation}
Then the arguments above can be concluded into two estimates, which we will use in our backward algorithm. The first one is 
\begin{equation}
\label{backward-l2-1}
    \sum_{O\in\co^{s-1}}\sum_{F_j\in\cf_O}\|g_{j,tang}\|_2^2\lesssim d^{-1} v_{s-1} \sum_{O\in\co^{s-1}}\|g_{O,tang}\|_2^2.
\end{equation}
The second one is, recalling \eqref{plane-tube-incidence},
\begin{equation}
\label{backward-l1-1}
    \big|N_{Rr_{s-1}^{-1/2}}L_{a,O}\big|\lesssim R^{O(\de)}R^2r^{-1/2}d^{3}v_{s-1}^{-1}. 
\end{equation}

${\bf 2.~Trans:}$ Suppose that at the step $s-1$, we were in the transverse case. Then for any $O_{s-1}\in\co^{s-1}$, there are $O(N_{s-1})$ many sets $F_j\in\cf$ that are the offspring of $O_{s-1}$. The number $N_{s-1}$ was defined in Lemma \ref{one-step-partitioning}. Similar to the cell case above, for any $O\in\co^{s-1}$, we can define a sum of wave packets $g_{O,tang}$ as  
\begin{equation}
    g_{O,tang}=\sum_{I_{s-1}\in  \ZI_{O,tang}}g_{I_{s-1}},
\end{equation}
where $\ZI_{O,tang}$ is the collection of bigger tubes of dimensions $\sim r_{s-1}^{1/2}\times r_{s-1}^{1/2}\times r_{s-1}$ defined as follows: For any bigger tube $I_{s-1}\in\ZI_{O,tang}$, there is at least one smaller tube $I_s\in\ZI_{j,tang}$ among all the offspring $F_j$ of $O$, such that $I_{s-1}\sim I_s$. Also, if we define the set
\begin{equation}
    L_{a,O}:=L_a\bigcap\big\{\cup_{\mr{I}\in\mr{\ZI}_{O,tang}}\mathring{I}\big\},
\end{equation}
we will have two estimates:
\begin{equation}
\label{backward-l2-2}
    \sum_{O\in\co^{s-1}}\sum_{F_j\in\cf_O}\|g_{j,tang}\|_2^2\lesssim {\rm Poly}(d)R^{-\de} \sum_{O\in\co^{s-1}}\|g_{O,tang}\|_2^2,
\end{equation}
and
\begin{equation}
\label{backward-l1-2}
    \big|N_{Rr_{s-1}^{-1/2}}L_{a,O}\big|\lesssim R^{O(\de)}R^2r^{-1/2}N_{s-1}. 
\end{equation}
The first estimate follows from a similar argument as the proof in Lemma \ref{transverse-lemma1}.

\subsection{General backward algorithm}\hfill

Generally, suppose that we are in the step $u\geq1$ of the iterated polynomial partitioning algorithm. We do the following:
\begin{enumerate}
    \item[{\bf Cell: }] If we have encountered the cell case in the step $u-1$, then we can find a dyadic number $v_{u-1}=O(d)$, so that for each cell $O\in\co^{u-1}$, we can define a collection of tubes $\ZI_{v_{u-1},O}$, a sum of wave packets $g_{O,tang}$, and a set
    \begin{equation}
        L_{a,O}:=L_a\bigcap\big\{\cup_{\mr{I}\in\mr{\ZI}_{v_{u-1},O}}\mathring{I}\big\}.
    \end{equation}
    Also, for each $O\in\co^{u-1}$, we define $\co^u_O\subset \co^u$ as the collection of all children of the set $O$. As we explained after the inequality \eqref{l2-backward-not-pigeonholing}, we have two estimates. The first one is
    \begin{equation}
    \label{backward-l2-3}
        \sum_{O\in\co^{u-1}}\sum_{O'\in\co^u_O}\|g_{O',tang}\|_2^2\lesssim d^{-1} v_{u-1} \sum_{O\in\co^{u-1}}\|g_{O,tang}\|_2^2.
    \end{equation}
    The second one is, for each $O\in\co^{u-1}$,
    \begin{equation}
    \label{backward-l1-3}
        \big|N_{Rr_{u-1}^{-1/2}}L_{a,O}\big|\lesssim d^{3}v_{u-1}^{-1}\max_{O'\in\co^u_O}|N_{Rr_{u}^{-1/2}}L_{a,O'}|. 
    \end{equation}
    \item[{\bf Trans: }] If we have encountered the transverse case in the step $u-1$, then similarly for each $O\in\co^{u-1}$, we can define a collection of tubes $\ZI_{O,tang}$, a sum of wave packets $g_{O,tang}$, a set
    \begin{equation}
        L_{a,O}:=L_a\bigcap\big\{\cup_{\mr{I}\in\mr{\ZI}_{O,tang}}\mathring{I}\big\},
    \end{equation}
    and a children set $\co^u_O$, so that 
    \begin{equation}
    \label{backward-l2-4}
        \sum_{O\in\co^{u-1}}\sum_{O'\in\co^u_O}\|g_{O',tang}\|_2^2\lesssim {\rm Poly}(d)R^{-\de} \sum_{O\in\co^{u-1}}\|g_{O,tang}\|_2^2,
    \end{equation}
    and for each $O\in\co^{u-1}$,
    \begin{equation}
    \label{backward-l1-4}
        \big|N_{Rr_{u-1}^{-1/2}}L_{a,O}\big|\lesssim N_{u-1}\max_{O'\in\co^u_O}\big|N_{Rr_{u}^{-1/2}}L_{a,O'}\big|. 
    \end{equation}
\end{enumerate}

We will use the convention $\co^0=\{B_R\}$, meaning that the largest set we would have in the backward algorithm is the ball $B_R$. In accordance to the inductive formula \eqref{tangent-inductive-2}, we define $f_{B_R,tang}$ to be the function $g_{B_R,tang}$, which only appears at the last step of the backward algorithm.

\vspace{3mm}

When the backward algorithm stops, we have a collection of tubes $\ZT_{B_R,tang}$, such that for each tube $T\in\ZT_{B_R,tang}$, the corresponding wave packet $f_T$ is summed in the function $f_{B_R,tang}$. We define
\begin{equation}
    v=\prod_{u}v_u,
\end{equation} 
be the product of $v_{u}$, for all the cell steps $u$. As a result, we can conclude from all the estimates in our backward algorithm, \eqref{backward-l2-1}, \eqref{backward-l2-2}, \eqref{backward-l2-3}, \eqref{backward-l2-4} and \eqref{backward-l1-1}, \eqref{backward-l1-2}, \eqref{backward-l1-3}, \eqref{backward-l1-4}, to obtain two estimates. The first one is
\begin{equation}
\label{backward-l2-final}
    \sum_{F_j\in\cf}\|g_{j,tang}\|_2^2\lesssim R^{O(\de)}d^{-s_c}R^{-s_t\de}v\sum_{T\in\ZT_{B_R,tang}}\|f_T\|_2^2.
\end{equation}
The second one is that, recalling \eqref{captical-f}, uniformly for all $|a|\leq 2R$,
\begin{equation}
\label{backward-l1-final}
    |N_{R^{1/2}}L_{a,B_R}|\lesssim R^{O(\de)}R^2r^{-1/2}|\cf|v^{-1},
\end{equation}
where
\begin{equation}
    L_{a,B_R}:=L_a\bigcap\big\{\cup_{T\in\ZT_{B_R,tang}}T\big\}.
\end{equation}

\subsection{Concluding the proof of Theorem \ref{broad-theorem}}\hfill

Finally, we will combine  \eqref{backward-l2-final}, \eqref{backward-l1-final} and the Littlewood-Paley theorem for translated cubes to conclude \eqref{broad-estimate}. For any $\theta\in\Theta$, we define 
\begin{equation}
    \ZT_{\theta,B_R,tang}:=\{T:T\in\ZT_{B_R,tang},~T{~\rm~dual~to~}\theta\}.
\end{equation}
For convenience, we will use $\ZT_\theta$ in short of $\ZT_{\theta,B_R,tang}$ in this subsection. As a consequence, we have 
\begin{equation}
    \sum_{T\in\ZT_{B_R,tang}}\|f_T\|_2^2\lesssim\sum_{\theta\in\Theta}\int|f_\theta|^2\Big(\sum_{T\in\ZT_{\theta}}\Id_T^\ast\Big).
\end{equation}

Let $B_\theta$ be an $R^{-1/2}$ cube containing $2\theta$ and let $\wh\vp_{B_\theta}$ be a smooth function supported in $2B_{\theta}$ with $\wh\vp_{B_\theta}(\xi)=1$ on $B_\theta$. If we set $f_{B_\theta}=\vp_{B_\theta}\ast f$, we get trivially $f_\theta=\vp_\theta\ast f_{B_\theta}$, implying that
\begin{equation}
    \sum_{\theta\in\Theta}\int|f_\theta|^2\Big(\sum_{T\in\ZT_{\theta}}\Id_T^\ast\Big)=\sum_{\theta\in\Theta}\int|\vp_\theta\ast f_{B_\theta}|^2\Big(\sum_{T\in\ZT_{\theta}}\Id_T^\ast\Big).
\end{equation}
By H\"older's inequality, $|\vp_\theta\ast f_{B_\theta}|^2\lesssim |\vp_\theta|\ast|f_{B_\theta}|^2$, which gives 
\begin{equation}
\label{sq-fcn-1}
    \sum_{\theta\in\Theta}\int|f_\theta|^2\Big(\sum_{T\in\ZT_{\theta}}\Id_T^\ast\Big)\lesssim \sum_{\theta\in\Theta}\int|f_{B_\theta}(y)|^2|\vp_\theta(x-y)|\Big(\sum_{T\in\ZT_{\theta}}\Id_T^\ast(x)\Big)dxdy.
\end{equation}

Let $T_\theta^0$ be the rectangular tube dual to $\theta$, centered at the origin. We introduce the Nikodym maximal function $Mf$,
\begin{equation}
    Mf(y)=\sup_{\theta\in\Theta}\frac{1}{|T_\theta^0|}\int\Id_{T_\theta^0}(x-y)f(x) dx,
\end{equation}
and the smooth Nikodym maximal function $M_sf$,
\begin{equation}
    M_sf(y)=\sup_{\theta\in\Theta}\int\vp_\theta(x-y)f(x) dx.
\end{equation}
Then, if we define the set $X$ as
\begin{equation}
    X=\bigcup_{\theta\in\Theta}\bigcup_{T\in\ZT_\theta}CT,
\end{equation}
from \eqref{sq-fcn-1} we have
\begin{equation}
    \sum_{\theta\in\Theta}\int|f_\theta|^2\Big(\sum_{T\in\ZT_{\theta}}\Id_T^\ast\Big)\lesssim \int\Big(\sum_{\theta\in\Theta}|f_{B_\theta}|^2\Big)M_s\Id_X.
\end{equation}
We invoke H\"older's inequality and the Littlewood-Paley theorem for translated cubes, Theorem \ref{L-P-cubes}, so that for $q$ being the H\"older's conjugate of $p/2$,
\begin{equation}
    \sum_{\theta\in\Theta}\int|f_\theta|^2\Big(\sum_{T\in\ZT_{\theta}}\Id_T^\ast\Big)\lesssim\Big\|\sum_{\theta\in\Theta}|f_{B_\theta}|^2\Big\|_{p/2}\|M_s\Id_X\|_q\lesssim\|f\|_p^2\|M_s\Id_X\|_q.
\end{equation}

\vspace{3mm}

Now we are in a position to apply \eqref{backward-l1-final}. Since $M\Id_X\leq1$, and since $M\Id_X$ is supported in $B_{2R}$, we use the support estimate \eqref{backward-l1-final} to obtain
\begin{equation}
    \|M\Id_X\|_q^q\lesssim R^{O(\de)}R^3\min\{1,r^{-1/2}|\cf|v^{-1}\}.
\end{equation}
Note that $\|M_s\Id_X\|_q\leq C_\e R^\be\|M\Id_X\|_q$. We combine this estimate with the two inequalities above so that
\begin{equation}
    \sum_{\theta\in\Theta}\int|f_\theta|^2\Big(\sum_{T\in\ZT_{\theta}}\Id_T^\ast\Big)\leq C_\e R^{O(\de)}(R^3\min\{1,r^{-1/2}|\cf|v^{-1}\})^{1/q}\|f\|_p^2.
\end{equation}
Plugging this back to \eqref{backward-l2-final}, we finally arrive at
\begin{equation}
\label{final-l2}
    \sum_{F_j\in\cf}\|g_{j,tang}\|_2^2\leq C_\e R^{O(\de)}d^{-s_c}R^{-s_t\de}v(R^3\min\{1,r^{-1/2}|\cf|v^{-1}\})^{1/q}\|f\|_p^2.
\end{equation}

Combining the fact $r=Rd^{-s_c}R^{-s_t\de}$, estimate \eqref{broad-tangent-square-fcn}, the fact that the quantities $\|\Id_{F_j\cap W}{\rm Bil}(g_{j,tang})\|_p$ are the same up to a constant factor,  estimate \eqref{cordoba-type-lemma}, estimate \eqref{final-l2} and a simple pigeonholing argument, we can conclude that 
\begin{equation}
\label{final-broad}
    \int_{B_R}|{\rm Br}_\al Sf|^p\lesssim_\e R^{O(\de)}|\cf|^\frac{2-p}{2}(d^{s_c}R^{s_t\de})^{p-3}v^\frac{p}{2}\min\{r^{1/2},|\cf|v^{-1}\})^\frac{p-2}{2}\|f\|_p^p.
\end{equation}

\vspace{3mm}
Our final task is to optimize \eqref{final-broad}. We expand the minimal function by considering two separate cases.
\subsubsection{Case $r^{1/2}>|\cf|v^{-1}$}\hfill

We rewrite the estimate \eqref{final-broad} as
\begin{equation}
\label{case-1-broad}
    \int_{B_R}|{\rm Br}_\al Sf|^p\leq C_\e R^{O(\de)}(d^{s_c}R^{s_t\de})^{p-3}v\|f\|_p^p.
\end{equation}
Recall that $|\cf|\lesssim d^{3s_c}$ and $v\lesssim d^{s_c}$. The constraint $r^{1/2}>|\cf|v^{-1}$ implies 
\begin{equation}
    R^{s_t\de}d^{7s_c}v^{-2}\leq R.
\end{equation}
We optimize \eqref{case-1-broad} by taking $v\sim d^{s_c}$, so that for $p\geq 3.25$,
\begin{equation}
    \int_{B_R}|{\rm Br}_\al Sf|^p\leq C_\e R^{\e}R^{p-3}\|f\|_p^p.
\end{equation}
This is \eqref{broad-estimate}.  \qed

\subsubsection{Case $r^{1/2}\leq|\cf|v^{-1}$}\hfill

We rewrite the estimate \eqref{final-broad} as
\begin{equation}
\label{case-2-broad}
    \int_{B_R}|{\rm Br}_\al Sf|^p\leq C_\e R^{O(\de)}|\cf|^\frac{2-p}{2}(d^{s_c}R^{s_t\de})^{p-3}v^\frac{p}{2}r^\frac{p-2}{4}\|f\|_p^p.
\end{equation}
We can similarly optimize the above inequalities by taking $v\sim d^{s_c}$ to conclude that for $p\geq3.25$, 
\begin{equation}
    \int_{B_R}|{\rm Br}_\al Sf|^p\leq C_\e R^{\e}R^{p-3}\|f\|_p^p. 
\end{equation}
This is again \eqref{broad-estimate}. \qed

\end{document}